\documentclass[11pt,a4paper,reqno]{amsart}

\usepackage{graphicx,mathrsfs,amssymb}
\usepackage[usenames]{color}
\usepackage[colorlinks=false,citecolor=blue]{hyperref}
\usepackage{enumerate}

\newtheorem{theorem}{Theorem}[section]
\newtheorem{lemma}[theorem]{Lemma}
\newtheorem{corollary}[theorem]{Corollary}
\newtheorem{proposition}[theorem]{Proposition}
\theoremstyle{definition}
\newtheorem{definition}[theorem]{Definition}
\newtheorem{remark}[theorem]{Remark}

\numberwithin{equation}{section}

\newcommand{\C}{{\mathbb C}}
\newcommand{\Z}{{\mathbb Z}}
\newcommand{\T}{{\mathbb T}}

\newcommand{\N}{{\mathbb N}}
\newcommand{\R}{{\mathbb R}}
\newcommand{\bk}{\mathbf{k}}

\newcommand{\FT}{\mathscr{F}_{\T^n}}
\newcommand{\FR}{\mathscr{F}_{\R^n}}
\newcommand{\CT}{C^{\infty}(\T^n)}

\newcommand{\card}{\mathop{\textrm{\upshape card}}}

\newcommand{\dbar}{{\textstyle d^{\hspace*{-0.4em}-}\hspace*{-0.15em}}}

\newcommand{\sumint}{\mathop{\mbox{$\displaystyle\sum\limits_{\bk\in\Z^n}\int_{\T^n}$}}}
\newcommand{\ossumint}{\mathop{\mbox{$\displaystyle\mathrm{Os}\!-\!\!\!\!\sum\limits_{\bk\in\Z^n}\int_{\T^n}$}}}
\newcommand{\supp}{\mathop{\textrm{supp}}}

\newcommand{\ph}{\varphi}
\newcommand{\op}{\mathop{\textrm{\upshape op}}}
\newcommand{\opp}[1]{A_{#1}}
\newcommand{\id}{\mathop{\textrm{\upshape id}}\nolimits}
\renewcommand{\Re}{\mathop{\textrm{Re}}}

\renewcommand{\epsilon}{\varepsilon}
\renewcommand{\bar}[1]{\overline{#1}}
\renewcommand{\tilde}[1]{\widetilde{#1}}

\renewcommand{\textmd}[1]{\color{black}#1\color{black}}

\allowdisplaybreaks
\begin{document}
\title[Pseudodifferential operators on the torus ]{Generation of semigroups for vector-valued pseudodifferential operators on the torus}
\author{B. Barraza Mart\'{i}nez}
\address{B. Barraza Mart\'{i}nez, Universidad del Norte, Departamento de Matem\'aticas,  Barranquilla (Colombia)}
\email{bbarraza@uninorte.edu.co}

\author{R.\ Denk}
\address{R.\ Denk, Universit\"at Konstanz, Fachbereich f\"ur Mathematik und Statistik, Konstanz (Germany)}
\email{robert.denk@uni-konstanz.de}

\author{J. Hern\'andez Monz\'on}
\address{J. Hern\'andez Monz\'on, Universidad del Norte, Departamento de Matem\'aticas,  Barranquilla (Colombia)}
\email{jahernan@uninorte.edu.co}

\author{T.\ Nau}
\address{T.\ Nau, Universit\"at Konstanz, Fachbereich f\"ur Mathematik und Statistik, Konstanz (Germany)}
\email{tobias.nau@uni-konstanz.de}

\date{\today\\
The authors would like to thank COLCIENCIAS (Project 121556933488) and DAAD for the financial support.}

\begin{abstract}
We consider toroidal pseudodifferential operators with operator-valued symbols, their mapping properties  and the generation of analytic semigroups on vector-valued Besov and Sobolev spaces. We show that a parabolic toroiodal pseudodifferential operator generates an analytic semigroup on the Besov space $B_{pq}^s(\T^n,E)$ and on the Sobolev space $W_p^k(\T^n,E)$, where $E$ is an arbitrary Banach space, $1\le p,q\le \infty$, $s\in\R$ and $k\in\N_0$. For the \textmd{proof of the } Sobolev space result, we establish a uniform estimate on the kernel which is given as an infinite \textmd{parameter-dependent } sum. An application to abstract non-autonomous periodic pseudodifferential Cauchy problems gives the existence and uniqueness of classical solutions for such problems.
\end{abstract}

\maketitle

\keywords{\footnotesize
\textbf{Keywords:} Pseudodifferential operators, vector-valued Sobolev spaces, toroidal Fourier transform, generation of analytic semigroup.}

\subjclass{\footnotesize{\textbf{Mathematical subject classification}: 35S05, 47D06, 35R20.}}

\section{Introduction}

In this article, we investigate vector-valued toroidal pseudodifferential operators, their mapping properties, and generation of semigroups in Besov and Sobolev spaces. Pseudodifferential operators on the $n$-dimensional torus $\T^n = (\R/2\pi \Z)^n$ can be treated by the toroidal quantization, see the monograph of Ruzhansky and Turunen \cite{ruzhansky-turunen10} as a standard reference. Here, the Fourier series approach to periodic pseudodifferential operators allows a global quantization with covariable $\bk\in\Z^n$, in contrast to the standard quantization on closed manifolds which yields local symbols with covariable $\xi\in\R^n$. Starting from the 1970's, the theory of toroidal pseudodifferential operators was developed by, e.g.,  Agranovich \cite{agranovich79}, Amosov \cite{amosov88}, McLean \cite{mclean91}, Melo \cite{melo97}, Saranen--Wendland \cite{saranen-wendland87},  Ruzhansky--Turunen \cite{ruzhansky-turunen07, ruzhansky-turunen09}, Turunen \cite{turunen00}, and Turunen--Vainikko \cite{turunen-vainikko98}. The mapping properties of toroidal pseudodifferential operators in $L^p$-spaces were studied by Delgado \cite{delgado13}, Molahajloo--Shahla--Wong \cite{molahajloo-shahla-wong10}, Wong \cite{wong11}, Cardona \cite{cardona14} and others. Here, the symbols of the considered pseudodifferential operators were scalar-valued, while in the present paper we will consider operator-valued symbols. To our knowledge, up to now no general results on periodic pseudodifferential operators in the vector-valued situation are available.

Pseudodifferential operators with operator-valued symbols appear, for instance, if the unknown function contains an additional parameter and therefore the value of the symbol belongs to some function space in this parameter. This is the case in coagulation-fragmentation problems where the additional parameter describes the cluster size (see Amann--Walker \cite{amann-walker05} and the references therein). On the other hand, operator-valued symbols can be used to solve elliptic and parabolic problems in cylindrical domains where the unknown function is considered as a function of the cylindrical variable with values in a function space in the cross-section \textmd{of the cylinder}. This approach was used, e.g., by Denk--Nau \cite{denk-nau13}, Favini--Giudetti-Yakubov \cite{favini-giudetti-yakubov11}, Nau--Saal \cite{nau-saal12}, and Rabinovich \cite{rabinovich11}. For an application to the Stokes system, see Denk--Seiler \cite{denk-seiler15}.

In the (periodic) vector-valued case, results are easier to obtain and sharper if the underlying Banach space has the geometric property of being a UMD space. In this case, Fourier multiplier results by Arendt--Bu \cite{arendt-bu02, arendt-bu04} and Bu--Kim \cite{bu-kim04} can be applied to obtain continuity in vector-valued $L^p$-Sobolev spaces (see Denk--Nau \cite{denk-nau13} in the case of differential operators). However, the restriction to UMD spaces excludes natural state spaces as $L^1$, and therefore we consider periodic pseudodifferential operators with operator-valued symbols acting in arbitrary Banach spaces. This can also be seen as a continuation of our papers  \cite{barraza-denk-monzon14}, \cite{barraza-denk-monzon14a} where we considered the non-periodic vector-valued case.

In the present paper, \textmd{we state as a first main result  } the continuity of operator-valued toroidal pseudodifferential operators in Besov spaces (Theorem~\ref{3.16} below), after having introduced the setting and some properties of vector-valued distributions on the torus in Section~2. For the definition of vector-valued Besov spaces, we use the nowadays standard method of dyadic decomposition. As mentioned above, we have no restriction on the underlying Banach space $E$ and we include $p,q\in\{1,\infty\}$. From the continuity in Besov spaces, we obtain that parabolic pseudodifferential operators generate analytic semigroups in the Besov scale $B_{pq}^s(\T^n,E)$ (Corollary~\ref{4.3}). A much deeper question concerns the generation of analytic semigroups in the Sobolev space $W_p^k(\T^n,E)$. The main result of this paper, Theorem~\ref{4.10}, gives an affirmative answer. For the proof of the resolvent estimate, we cannot apply Michlin type results as stated in Arendt-Bu \cite{arendt-bu04} as we did not assume a UMD space. However, we could establish kernel estimates by a careful analysis of parameter-dependent sums (Lemma~\ref{4.7}) which are the key to prove the generation of an analytic semigroup in Sobolev spaces. Based on Theorem~\ref{4.10} and on results by Amann \cite{amann95}, we obtain well-posedness for abstract parabolic non-autonomous Cauchy problems in Sobolev spaces (Theorem~\ref{4.12}).

\section{Vector-valued distributions on the torus and the Fourier transform}

Throughout this article, $E$ stands for an arbitrary Banach space with norm $\|\cdot\|$, $n\in\N$ is fixed, $\langle x\rangle := (1+|x|^2)^{1/2}$ for $x\in\R^n$ where $|\cdot|$ denotes the euclidian norm on $\R^n$.

In the present section, we define the main spaces of vector-valued functions and distributions on the $n$-dimensional torus $\T^n := \textmd{(\R/2\pi\Z)^n}$, generalizing the scalar theory given, e.g., in Chapter~3 of \cite{ruzhansky-turunen10}. \textmd{As a set of representatives for $\T^n$, we choose $[-\pi,\pi]^n$. Note that in this case the distance of an element $x\in\T^n$ to the origin is given by the euclidian norm $|x|$.}

For $m\in\N_0\cup\{\infty\}$, let $C^m(\T^n,E)$ denote the space of all $m$-times continuously differentiable functions $\ph\colon \T^n\to E$. We remark that a function $\ph\colon\T^n\to E$ can be identified with a function $\ph\colon\R^n\to E$ which is $2\pi$-periodic in all variables. We will tacitly use this identification in the following. As usual, $C^{\textmd{\infty}}(\T^n,E)$ is endowed with the locally convex topology which is induced by the family of seminorms $\{q_N: N\in\N_0\}$ given by
\[ q_N(\ph) := \max_{|\alpha|\le N}\sup_{x\in\T^n} \|\partial^\alpha \ph(x)\|\quad (\ph\in C^\infty(\T^n,E)).\]
Here and in the following, we use the standard multi-index notation.
By this construction, $C^{\textmd{\infty}}(\T^n,E)$ becomes a Fr\'{e}chet space. We write $C^m(\T^n):= C^m(\T^n,\C)$ for $m\in\N_0\cup\{\infty\}$.

The space $\mathscr D'(\T^n,E):= L(C^\infty(\T^n), E)$ is called the space of all $E$-valued toroidal distributions and is endowed with the weak-*-topology induced by \textmd{the family of  seminorms $\{\tilde q_\ph: \ph\in C^\infty(\T^n)\}$ with }
\[ \tilde q_\ph(u) := \|\langle u,\ph\rangle\|\quad (u\in\mathscr D'(\T^n,E)).\]
Here we have written $\langle u,\ph\rangle := u(\ph)$.

On the dual space $\Z^n$ of $\T^n$, we will consider the space $\mathscr S(\Z^n,E)$ of rapidly decreasing functions, i.e. of functions $\ph\colon\Z^n\to E$ such that for every $N\in\N_0$ we have
\[ p_N(\ph) := \sup_{\bk\in\Z^n} \langle \bk\rangle ^N \|\ph(\bk)\| <\infty.\]
In a natural way, $\mathscr S(\Z^n,E)$ is a Fr\'{e}chet space again. We define $\mathscr S'(\Z^n,E):= L(\mathscr S(\Z^n),E)$, the space of $E$-valued tempered distributions, again endowed with the weak-*-topology. We say that a function $f\colon\Z^n\to E$ has at most polynomial growth if there exist constants $C\ge 0$ and $M\in\N_0$ such that
\[ \|f(\bk)\| \le C\langle \bk\rangle ^M \quad (\bk\in\Z^n).\]
The space of all such functions is denoted by $\mathcal O(\Z^n,E)$. As in the scalar case (see \cite{ruzhansky-turunen10}, Exercise 3.1.7), one easily gets the following identification.

\begin{lemma}
  \label{2.1}
  For $u\in\mathscr S'(\R^n,E)$ define $\textmd{g}_u(\bk) := \langle u,\delta_{\bk}\rangle\;(\bk\in\Z^n)$ where
  \[\delta_{\bk}({\boldsymbol \ell}) := \delta_{\bk\boldsymbol\ell} := \begin{cases}
    1,& \bk = \boldsymbol\ell,\\ 0, & \bk\not=\boldsymbol\ell.
  \end{cases}\]
  Then the linear map $u\mapsto \textmd{g}_u,\; \mathscr S'(\Z^n,E)\to \mathcal O(\Z^n,E)$ is bijective. In particular, we have for all $u\in\mathscr S'(\Z^n,E)$
  \begin{equation}
    \label{2-1}
    \langle u,\ph\rangle = \sum_{\bk\in\Z^n} \ph(\bk) \textmd{g}_u(\bk)\quad (\ph\in\mathscr S(\Z^n)),
  \end{equation}
 \textmd{ where the sum on the right-hand side is absolutely convergent.}
\end{lemma}

\begin{proof}
  Let $u\in\mathscr S'(\Z^n,E)$. By continuity of $u$, there exist $M\in\N_0$ and $C>0$ such that
  \[ \|\textmd{g}_u(\bk)\|=\|\langle u,\delta_{\bk}\rangle\| \le Cp_M(\delta_\bk) = C \langle\bk\rangle^M\quad (\bk\in\Z^n).\]
  This shows $\textmd{g}_u\in\mathcal O(\Z^n,E)$. Now we easily obtain \eqref{2-1}, where the convergence of the series follows from $\textmd{g}_u\in\mathcal O(\Z^n,E)$ and $\ph\in \mathscr S(\Z^n)$. Due to \eqref{2-1}, the map $u\mapsto \textmd{g}_u$ is injective. Finally, a straightforward computation shows that for every $\textmd{g}\in \mathcal O(\Z^n,E)$ the definition
  \[ \textmd{\langle u_{\textmd{g}},\ph\rangle} := \sum_{\bk\in\Z^n} \ph(\bk) \textmd{g}(\bk)\quad (\ph\in\mathscr S(\Z^n))\]
  gives an element $u_{\textmd{g}}\in\mathscr S'(\Z^n,E)$ which shows surjectivity of this map.
\end{proof}

The last lemma shows that tempered distributions on $\Z^n$ are in fact functions, in contrast to the $\R^n$-case.

For $f\in C^\infty(\T^n,E)$, we define the toroidal Fourier transform of $f$, denoted by $\FT f$,  as
 \begin{equation}\label{2-2}
 (\FT f)(\bk):=\int_{\T^n}e^{-i\bk \cdot x}f(x)\,\dbar x = \int_{\textmd{[-\pi,\pi]^n}}e^{-i\bk \cdot x}f(x)\,\dbar x \quad (\bk\in\Z^n),
\end{equation}
where $\dbar x:=(2\pi)^{-n} dx$.
Note that \textmd{the }  integral in \eqref{2-2} is understood as a Bochner integral and its value is an element of $E$.

In the same way as in the scalar case, one can show the following results.

\begin{lemma}
  \label{2.2}
  a) The Fourier transform $\FT\colon C^\infty(\T^n,E)\to \mathscr S(\Z^n,E)$ is an isomorphism of Fr\'{e}chet spaces, i.e. a linear and continuous bijection with continuous inverse. Its inverse is given by
  \[ (\FT^{-1} \ph)(x) = \sum_{\bk\in\Z^n} e^{i\bk\cdot x} \ph(\bk)\quad (x\in\T^n,\, \ph\in\mathscr S(\Z^n,E)).\]

  b) Let $D^\alpha := (-i)^{|\alpha|}\partial^\alpha$. Then
  \[ \big[\FT(D^\alpha\ph)](\bk) = \bk^\alpha (\FT\ph)(\bk)\quad (\bk\in\Z^n)\]
  for all $\alpha\in\N_0^n$ and $\ph\in C^\infty(\T^n,E)$.
\end{lemma}

For $f\in C^\infty(\T^n,E)$, the associated regular distribution $u_f\in\mathscr D'(\T^n,E)$ is given by
\[ \textmd{\langle u_f,\ph\rangle} := \int_{\T^n} \ph(x) f(x)\dbar x\quad (\ph\in C^\infty(\T^n)).\]
By  Lemma~\ref{2.1}, $\FT f\in\mathscr S(\Z^n,E)$ can be considered as an element of $\mathscr S'(\Z^n,E)$ as it is polynomially bounded.  For $\ph\in\mathscr S(\Z^n)$ we obtain
\begin{align*}
  \langle \FT f,\ph\rangle & = \sum_{\bk\in\Z^n} \ph(\bk)(\FT f)(\bk) = \sum_{\bk\in\Z^n} \int_{\T^n} e^{-i\bk\cdot x}\ph(\bk)f(x)\,\dbar x\\
  & = \int_{\T^n} \Big(\sum_{\bk\in\Z^n} \ph(\bk) e^{-i\bk\cdot x}\Big)f(x)\,\dbar x\\
  & = \int_{\T^n} (\FT^{-1}\ph)(-x)f(x)\,\dbar x = \big\langle u_f, (\FT^{-1}\ph)(-\,\cdot\,)\big\rangle.
\end{align*}
This motivates the definition of the  Fourier transform of a toroidal  distribution.

\begin{definition}\label{2.3}
For $u\in\mathscr {D}'(\T^n,E)$, the Fourier transform of $u$ is defined by
\[
\langle \FT u , \varphi \rangle := \bigl\langle u , (\FT^{-1}\varphi)(-\,\cdot\,) \bigr\rangle\quad (\varphi\in\mathscr S(\Z^n)).
\]
\end{definition}

\begin{remark}\label{2.4}
By Lemma~\ref{2.2} a) and the definition of the topologies, one immediately sees that
\[ \FT\colon \mathscr D'(\T^n,E)\to \mathscr S'(\Z^n,E)\]
is linear, continuous and bijective with continuous inverse. For $v\in\mathscr S'(\Z^n,E)$, the inverse is given by
\[ \textmd{\langle \FT^{-1} v,\ph\rangle =} \big\langle v,(\FT\ph)(-\,\cdot\,)\big\rangle \quad (\ph\in C^\infty(\T^n)).\]
\end{remark}

\begin{definition}
  \label{2.5}
  For $u\in\mathscr D'(\T^n,E)$, we define the Fourier coefficients $(\hat u(\bk))_{\bk\in\Z^n}$ by
  \[ \hat u(\bk) := \textmd{\langle u, e_{-\bk}\rangle} \quad (\bk\in\Z^n),\]
  where $e_{-\bk}\in C^\infty(\T^n)$ is given by $e_{-\bk}(x) := e^{-i\bk\cdot x}\;(x\in\T^n)$.
\end{definition}

\begin{remark}
  \label{2.6}
  Let $u\in\mathscr D'(\T^n,E)$.
  Identifying again $\mathscr S'(\Z^n,E)$ and $\mathcal O(\Z^n,E)$, we can write
  \begin{align*}
    \hat u(\bk) & = \textmd{\langle u, e_{-\bk}\rangle} = \textmd{\big\langle \FT^{-1}\FT u,e_{-\bk}\big\rangle}
    = \big\langle \FT u, (\FT e_{-\bk})(-\,\cdot\,)\big\rangle \\
    & = \sum_{\boldsymbol \ell\in\Z^n} (\FT e_{-\bk})(-\boldsymbol\ell) (\FT u)(\boldsymbol \ell) = (\FT u)(\bk),
  \end{align*}
  where we used in the last step that $(\FT e_{-\bk})(-\boldsymbol\ell) =\int_{\T^n} e^{i(\boldsymbol\ell-\bk)\cdot x} \dbar x = \delta_{\bk,\boldsymbol\ell}$.
  Therefore, the $\bk$-th Fourier coefficient $\hat u(\bk)$ is given by the value of $\FT u$, considered as a function on $\Z^n$ with polynomial growth, at the point $\bk\in\Z^n$. Note that this holds, in particular, for $u\in C^\infty(\T^n,E)$, considered as a regular distribution.
\end{remark}

\begin{definition}
  \label{2.7}
For $\psi\in\CT$, $v\in\mathscr {D}'(\T^n)$ and $z\in E$, the tensor products $\psi\otimes z$ and $v\otimes z$ are defined by
\begin{align*}
  (\psi\otimes z)(x) & := \psi(x)z\quad (x\in\T^n),\\
  (v\otimes z)(\ph) & := \textmd{\langle v,\ph\rangle} z\quad (\ph\in C^\infty(\T^n)).
\end{align*}
\end{definition}

It is straightforward to prove that $\psi \otimes z\in C^\infty(\T^n,E)$ and $v\otimes z\in\mathscr D'(\T^n,E)$. The following result shows that the series appearing in the inversion \textmd{formulae } (see Lemma~\ref{2.2} a) and Remark~\ref{2.4}, respectively) also converge in the corresponding locally convex topology.

\begin{lemma}
  \label{2.8}
  a) For all $f\in C^\infty(\T^n,E)$ we have
  \begin{equation}
    \label{2-3}
    f = \sum_{\bk\in\Z^n} e_\bk \otimes \hat f(\bk) \quad \text{ in } C^\infty(\T^n,E).
  \end{equation}

  b)   For all $u\in \mathscr D'(\T^n,E)$ we have (identifying $e_\bk$ and the induced regular distribution)
  \begin{equation}
    \label{2-4}
    u = \sum_{\bk\in\Z^n} e_\bk\otimes \hat u(\bk)\quad \text{ in }\mathscr D'(\T^n,E).
  \end{equation}
  \end{lemma}

\begin{proof}
a) Let $f\in C^\infty(\T^n,E)$. For $R>0$, define $f_R := \sum_{|\bk|\le R} e_\bk\otimes \hat f(\bk)$. We have to show that
\[ p_N(f-f_R) = \max_{|\alpha|\le N} \sup_{x\in\T^n} \big\| D^\alpha (f-f_R)\textmd{(x)}\big\| \to 0 \quad (R\to\infty)\]
for all $N\in\N_0$. For this, we apply Lemma~\ref{2.2} and write
\[ D^\alpha f(x) = \sum_{\bk\in\Z^n} e^{i\bk\cdot x} [\FT (D^\alpha f)](\bk) = \sum_{\bk\in\Z^n} e^{i\bk\cdot x} \bk^\alpha \hat f(\bk).\]
In the same way, we have $D^\alpha (e_\bk\otimes \hat f(\bk))(x) = e^{i\bk\cdot x} \bk^\alpha \hat f(\bk)$. Therefore,
\begin{align*}
  p_N(f-f_R) & = \max_{|\alpha|\le N} \sup_{x\in\T^n} \Big\| \sum_{|\bk|>R} e^{i\bk\cdot x} \bk^\alpha \hat f(\bk)\Big\| \\
  & \le \max_{|\alpha|\le N} \sum_{|\bk|>R} |\bk^\alpha| \,\|\hat f(\bk)\| \le \sum_{|\bk|>R} \langle \bk\rangle^N \|\hat f(\bk)\|.
\end{align*}
As $\FT f\in\mathscr S(\Z^n,E)$, the last sum converges to zero for $R\to\infty$ for every fixed $N\in\N_0$.

b) Let $u\in\mathscr D'(\T^n,E)$, $\ph\in C^\infty(\T^n)$ and set again $\ph_R := \sum_{|\bk|\le R} e_\bk\otimes \hat \ph(\bk)$ for $R>0$. For $\bk\in\Z^n$ we obtain
\[ \textmd{\langle u,e_\bk\otimes \hat \ph(\bk)\rangle}  = \textmd{\langle u,\hat\ph(\bk)e^{i\bk\,\cdot\,}\rangle}= \hat\ph(\bk) \textmd{\langle u,e^{i\bk\,\cdot\,}\rangle}  = \hat\ph(\bk)\hat u(-\bk).\]
On the other hand, due to
\[ \textmd{\langle e_\bk,\ph\rangle}  = \int_{\T^n} e^{ix\cdot\bk} \ph(x)\dbar x = (\FT \ph)(-\bk),\]
we see that $\textmd{\langle e_\bk\otimes \hat u(\bk),\ph\rangle} = \hat\ph(-\bk)\hat u(\bk)$. This gives
\[ \textmd{\sum_{|\bk|\le R} \langle e_\bk\otimes\hat u(\bk),\ph\rangle = \sum_{|\bk|\le R} \hat\ph(-\bk)\hat u(\bk) = \sum_{|\bk|\le R} \hat\ph(\bk) \hat u(-\bk) = \langle u, \ph_R\rangle.}\]
Taking now $R\to\infty$, the statement follows from part a) and the continuity of $u$.
\end{proof}

We are now going to consider pseudodifferential operators on the torus.
For an operator-valued function $a\in \mathcal O(\Z^n,L(E))$ and $v\in \mathcal O(\Z^n,E)$, we define
\[ \langle av,\ph\rangle := \sum_{\bk\in\Z^n} \ph(\bk)a(\bk)v(\bk)\quad (\ph\in\mathscr S(\Z^n)).\]
Due to Lemma~\ref{2.1} and the fact that products of polynomially bounded functions are again polynomially bounded, we see that $av\in\mathscr S'(\Z^n,E)$. Therefore, for all $u\in\mathscr D'(\T^n,E)$ and $a\in\mathcal O(\Z^n,L(E))$, we obtain $a \FT u\in \mathscr S'(\Z^n,E)$, and the following definition makes sense.

\begin{definition}
  \label{2.9}
  For $a\in\mathcal O(\Z^n,L(E))$, we define $\textmd{\op[a]}\colon \mathscr D'(\T^n,E)\to \mathscr D'(\T^n,E)$ by
  \[  \textmd{\op[a]u} := \FT^{-1} \big( a\FT u)\quad (u\in \mathscr D'(\T^n,E)).\]
  We call $\textmd{\op[a]}$ the toroidal pseudodifferential operator associated to the discrete symbol $a$. If we want to distinguish the toroidal case from the whole space case, we write more precisely $\op_{\T^n} [a]$.

\textmd{  If a function $\tilde a\colon \R^n\to L(E)$ with $\tilde a|_{\Z^n} \in \mathcal O(\Z^n,L(E))$ is given, we will also write $\op[\tilde a]$ instead of $\op[\tilde a|_{\Z^n}]$.
}
\end{definition}

\begin{lemma}
  \label{2.10}
  a) For $a\in\mathcal O(\Z^n,L(E))$, we have \textmd{$\op[a]\in L(C^\infty(\T^n,E))$.}

  b)  For $a\in\mathcal O(\Z^n,L(E))$ and $u\in\mathscr D'(\T^n,E)$, we have
  \[ \textmd{\op[a]} u = \sum_{\bk\in\Z^n} e_\bk \otimes a(\bk)\hat u(\bk)\quad \text{ in }\mathscr D'(\T^n,E).\]
\end{lemma}

\begin{proof}
  a) Let $a\in\mathcal O(\Z^n,L(E))$. Then there exist $C\ge 0$ and $M\in\N_0$ such that $\|a(\bk)\|_{L(E)}\le C\langle \bk\rangle^M$. For $v\in \mathscr S(\Z^n,E)$ and  $N\in\N_0$,
  \begin{align*}
   p_N(av) & = \sup_{\bk\in\Z^n} \langle\bk\rangle^N\|a(\bk)v(\bk)\| \\
   & \le \Big( \sup_{\bk\in\Z^n} \langle\bk\rangle^{-M} \|a(\bk)\|_{L(E)} \Big) \Big( \sup_{\bk\in\Z^n} \langle\bk\rangle^{M+N} \|v(\bk)\|\Big) \le C p_{M+N}(v).
   \end{align*}
  This shows  that the map $v\mapsto av,\, \mathscr S(\Z^n,E)\to \mathscr S(\Z^n,E)$ is continuous. From the continuity of the Fourier transform $\FT$ and its inverse \textmd{(Lemma~\ref{2.2} a))}, we obtain that \textmd{$\op[a]\colon C^\infty(\T^n,E)\to C^\infty(\T^n,E)$ } is (well-defined and) continuous.

  b) follows from Lemma~\ref{2.8} b) and the fact that $a\FT u\in\mathscr S'(\textmd{\Z^n},E)$.
\end{proof}

\section{Pseudodifferential operators on toroidal Besov spaces}

In this section, we consider a class of (toroidal) pseudodifferential operators in the setting of vector-valued Besov spaces. In analogy to the $\R^n$-case, toroidal pseudodifferential operators are defined by conditions on their discrete derivatives, using the difference operator. We refer to \cite{ruzhansky-turunen10}, Section~3.3.1, for a more detailed exposition of the discrete analysis toolkit, and will only summarize the main ingredients.

For $j\in\{1,\dots,n\}$, we denote the $j$-th unit vector in $\R^n$ by $\delta_j :=(\delta_{jk})_{k=1,\dots,n}$ where $\delta_{jk}$ stands for the Kronecker symbol. For $a\colon\Z^n\to E$, $\bk =(k_1,\dots,k_n)^\top\in \Z^n$ and $\alpha\in\N_0^n$, the discrete differences are defined by
\begin{align*}
  \Delta_{k_j} a(\bk) & := a(\bk +\delta_j) - a(\bk),\\
  \bar \Delta_{k_j} a(\bk) & := a(\bk)-a(\bk-\delta_j),\\
  \Delta_{\bk}^\alpha & := \Delta_{k_1}^{\alpha_1}\ldots\Delta_{k_n}^{\alpha_n} ,\\
   \bar \Delta_\bk^\alpha & :=\bar \Delta_{k_1}^{\alpha_1}\ldots\bar \Delta_{k_n}^{\alpha_n} .
\end{align*}

The following definition of the symbol class is similar to the standard definition (see, e.g., Definition~\textmd{II.4.1.7 } in \cite{ruzhansky-turunen10}). However, we restrict ourselves to the standard H{\"o}rmander class $S_{1,0}^m$ and to $x$-independent symbols. On the other hand, our symbols are vector-valued and we need only finitely many conditions on the differences.

\begin{definition}
  \label{3.1} For $m\in\R$ and $\rho\in\N_0$, the symbol class $S^{m,\rho} = S^{m,\rho}_{1,0}(\Z^n,L(E))$ consists of all functions $a\colon\Z^n\to L(E)$ such that for all $\alpha\in\N_0^n$ with $\vert\alpha\vert\le\rho$, there exists $C_\alpha\ge 0$ with
\begin{equation}\label{def-ineq-symbolclass}
 \Vert \Delta_\bk^\alpha a(\bk)\Vert_{{L}(E)} \le C_\alpha\langle \bk\rangle^{m-\vert\alpha\vert} \quad (\bk\in\Z^n).
\end{equation}
 In $S^{m,\rho}$ we define the norm
\begin{equation}\label{def-norm-symbolclass}
 \Vert a\Vert_{S^{m,\rho}} := \max_{\vert\alpha\vert \le \rho} \,\,\sup_{\bk\in\Z^n}\, \langle \bk\rangle^{\vert\alpha\vert -m}\Vert \Delta_\bk^\alpha a(\bk)\Vert_{{L}(E)}.
\end{equation}
\end{definition}

\begin{remark}
  \label{3.2}
  For all $m\in\R$ and $\rho\in\N_0$, we have $S^{m,\rho}\subset \mathcal O(\Z^n,L(E))$, and therefore, $\textmd{\op[a]}$ is well defined for $a\in S^{m,\rho}$. For $a\in S^{m,\rho}$ with $m<-n$ \textmd{and $f\in C^\infty(\T^n,E)$}, we can write
  \begin{align*}
    \big(\textmd{\op[a]}f\big)(x) & = (\FT^{-1} a \FT f)(x) = \sum_{\bk\in\Z^n}\textmd{\int_{\T^n}} e^{i\bk\cdot (x-y)} a(\bk) f(y) \dbar y \\
    & = \sum_{\bk\in\Z^n} \int_{\T^n} e^{i\bk\cdot y} a(\bk) f(x-y)\dbar y\quad (x\in\T^n).
  \end{align*}
  Note that for general $m$ the sum does not converge absolutely. Therefore, we consider an oscillatory version in analogy to the continuous case.
\end{remark}

\begin{definition}
  \label{3.3}
  Let $m\in\R$, $\rho\in\N_0$, and let $\chi\in\mathscr S(\R^n)$ with $\chi(0)=1$. For $a\in S^{m,\rho}$ and $f\in C^\infty(\T^n,E)$, we define
  \[ \ossumint e^{i\bk\cdot y} a(\bk) f(x-y)\dbar y := \lim_{\epsilon\searrow 0} \sumint e^{i\bk \cdot y} \chi(\epsilon\bk) a(\bk) f(x-y)\dbar y\]
  for all $x\in\T^n$.
\end{definition}

\begin{remark}\label{3.4}
a)   Note that the sum in Definition~\ref{3.3} is absolutely convergent. By integration by parts and dominated convergence, we easily see that for $N\in\N$ with $2N>m+n$
  \begin{align*}
    \textmd{\big(\op[a] f\big)}(x) & = \sumint e^{i\bk\cdot y} \langle \bk\rangle^{-2N} a(\bk) \big[(1-\mathcal L_y)^N f\big](x-y)\dbar y \\
    & = \ossumint e^{i\bk\cdot y} a(\bk) f(x-y) \dbar y,
  \end{align*}
  where $\mathcal L_y$ denotes the Laplacian on $\T^n$ with respect to the variable $y$. In particular, the definition of the oscillatory integral in Definition~\ref{3.3} does not depend on $\chi$.

  b) As the symbols we consider are $x$-independent, the symbol of the composition of two pseudodifferential operators equals the product of the symbols. More precisely, let $\rho\in\N_0$, $m_i\in\R$ and $a_i\in S^{m_i,\rho}$ for $i=1,2$. Then it follows directly from the definitions that $a_1a_2\in S^{m_1+m_2,\rho}$ and $\textmd{\op[a_1a_2]=\op[a_1]\op[a_2]}$ as operators both in $C^\infty(\T^n,E)$ and in $\mathscr D'(\T^n,E)$. Here $(a_1a_2)(\bk):= a_1(\bk)a_2(\bk)$ is the composition of the operators in $L(E)$ for every $\bk\in\Z^n$.
\end{remark}

In the following, we will consider vector-valued Besov and Sobolev spaces on the torus. We follow a standard approach based on dyadic decompositions which in the  $\R^n$-case  can be found, e.g., in \cite{triebel78}, Section~2.3.

\begin{definition}\label{3.5}
  A sequence $(\phi_j)_{j\in\N_0}\subset\mathscr S(\R^n)$ is called a dyadic decomposition if $\supp\phi_0\subset\{\xi\in\R^n: |\xi|\le 2\}$, $\supp\phi_j\subset\{\xi\in\R^n: 2^{j-1}\le|\xi|\le 2^{j+1}\}$, $\sum_{j=0}^\infty \phi_j(\xi) =1 \;(\xi\in\R^n)$, and if for each $\alpha\in\N_0^n$ there exists a constant $c_\alpha>0$ such that
  \[ |D_\xi^\alpha \phi_j(\xi)| \le c_\alpha 2^{-j|\alpha|}\quad (\xi\in\R^n,\, j\in\N_0).\]
\end{definition}

\textmd{\begin{remark}\label{3.5a}
A dyadic decomposition can be constructed in the following way: Let $\phi_0\in\mathscr S(\R^n)$ with $\supp\phi_0\subset\{\xi\in\R^n: |\xi|\le \frac32\}$ and $\phi_0(\xi)=1\;(|\xi|\le 1)$. Then we define $\phi_j(\xi) := \phi_0(2^{-j}\xi) - \phi_0(2^{-j+1}\xi) $ for $j\in\N$ and $\xi\in\R^n$.
\end{remark}
}
In the following, let $(\phi_j)_{j\in\N_0}$ be a dyadic decomposition. Additionally, we may assume that $\phi_j$ is constructed in the above way.

\begin{lemma}
  \label{3.6}
  a) For $\psi\in\mathscr S(\Z^n)$, we have $\sum_{j=0}^\infty \phi_j|_{\Z^n}\psi=\psi$ in $\mathscr S(\Z^n)$.

  b) For $u\in \mathscr D'(\T^n,E)$, we have $\sum_{j=0}^\infty \op[\phi_j]u = u$ in $\mathscr D'(\T^n,E)$.
\end{lemma}

\begin{proof}
  a) By the properties of a dyadic decomposition, we have $|\phi_j(\xi)|\le c_0\; (j\in\N_0,\, \xi\in\R^n)$ and, noting the conditions on the support, $\sum_{j=0}^\infty |\phi_j(\xi)|\le \textmd{2}c_0\;(\xi\in\R^n)$.

  Let $\psi\in\mathscr S(\R^n)$. For $\epsilon>0$ and $N\in\N_0$, there exists $R>0$ such that
  \[ \textmd{\langle \bk\rangle^N} |\psi(\bk)| <\frac{\epsilon}{1+\textmd{2}c_0}\quad (\bk\in\Z^n\;\text{ with }|\bk|\textmd{\ge} R).\]
  Again due to the support condition on $(\phi_j)_{j\in\N_0}$, there exists $m_0\in\N$ with $\sum_{j=0}^m \phi_j(\bk) = \sum_{j=0}^\infty \phi_j(\bk) = 1$ for all $m\ge m_0$ and $\bk \in\Z^n$ with $|\bk |\le R$. Then for all $m\ge m_0$ we obtain
  \begin{align*}
    p_N\Big( \sum_{j=0}^m \phi_j\textmd{\big|_{\Z^n}}\psi -\psi\Big) & = \sup_{\bk\in\Z^n} \langle \bk\rangle ^N
    \Big| \Big( \sum_{j=0}^m \phi_j(\bk) -1\Big)\psi(\bk)\Big| \\
    & \le \sup_{|\bk|\ge R} \langle\bk\rangle^N \Big( 1+ \sum_{j=0}^m|\phi_j(\bk)|\Big) |\psi(\bk)| \\
    & \le (1+\textmd{2} c_0)\tfrac{\epsilon}{1+\textmd{2}c_0} = \epsilon.
  \end{align*}
  Thus for all $N\in\N$ we see that $p_N(\sum_{j=0}^m\phi_j\textmd{|_{\Z^n}}\psi -\psi)\to 0\;(m\to\infty)$, which implies a).

  b) Let \textmd{$u\in\mathscr D'(\T^n,E)$, } $\ph\in C^\infty(\T^n)$ and $\psi:=(\FT \ph)(-\,\cdot\,)$. Then $\psi\in\mathscr S(\Z^n)$, and by a) and the continuity of $\textmd{\FT u}$ on $\mathscr S(\Z^n)$ we obtain
  \begin{align*}
    \Big\| \Big\langle \sum_{j=0}^m & \FT^{-1} (\phi_j\textmd{|_{\Z^n}}\FT u)-u , \ph \Big\rangle\Big\| \\
    & = \Big\| \Big\langle \sum_{j=0}^m \phi_j\textmd{|_{\Z^n}}\FT u -\FT u , (\FT\ph)(-\,\cdot\,) \Big\rangle \Big\| \\
    & = \Big\| \sum_{j=0}^m \big\langle \FT u , \phi_j\textmd{|_{\Z^n}} \psi \big\rangle - \big\langle \FT u , \psi\big\rangle \Big\| \\
    & = \Big\| \Big\langle \FT u ,  \sum_{j=0}^m \phi_j\textmd{|_{\Z^n}}\psi -\psi \Big\rangle \Big\| \to 0\quad (m\to\infty).
  \end{align*}
  By definition of the weak-*-topology, this implies $\sum_{j=0}^m \op[\phi_j]u \to u$ in $\mathscr D'(\T^n,E)$.
\end{proof}

\begin{definition}
  \label{3.7}
  Let $p,q\in[1,\infty]$, $m\in\N_0$, and $s\in\R$.

  a) The toroidal Sobolev space $W_p^m(\T^n,E)$ is defined as the set of all $u\in L^p(\T^n,E)$ for which $D^\alpha u\in L^p(\T^n,E)\;(|\alpha|\le m)$. Here, for $p\in [1,\infty)$ the space $L^p(\T^n,E)$ is the space of all (equivalence classes of) strongly measurable functions $u\colon\T^n\to E$ with $\|u\|_{L^p(\T^n,E)} := (\int_{\T^n} \|f(x)\|^p \dbar x)^{1/p}\textmd{<\infty}$. The norm in $W_p^m(\T^n,E)$ is defined as
  \[ \|u\|_{W_p^m(\T^n,E)} := \Big(\sum_{|\alpha|\le m} \|D^\alpha u\|^{\textmd{p}}_{L^p(\T^n,E)}\Big)^{1/p}\quad (u\in W_p^m(\T^n,E)).\]
  For $p=\infty$ we have the usual modification.

  b) The toroidal Besov space $B^s_{pq}(\T^n,E)$ is defined as $B^s_{pq}(\T^n,E):= \{u\in\mathscr D'(\T^n,E): \|u\|_{B^s_{pq}(\T^n,E)}<\infty\}$ with
  \[ \|u\|_{B_{pq}^s(\T^n,E)} := \Big\| \Big( 2^{js} \big\| \op[\phi_j]u\big\|_{L^p(\T^n,E)}\Big)_{j\in\N_0}\Big\|_{\ell^q(\N_0)}.\]
\end{definition}

\begin{remark}
  \label{3.8}
  In the same way as in the vector-valued continuous ($\R^n$-)case (see \cite{amann09}, Section~3.3), one sees the following properties: \textmd{Different choices of the dyadic decomposition lead to equivalent norms, } and for all $m\in\N_0$ and $p\in [1,\infty]$, we have
  \begin{gather*}
    B_{p1}^m(\T^n,E) \hookrightarrow W_p^m(\T^n,E) \hookrightarrow B_{p\infty}^m(\T^n,E),\\
    B_{\infty1}^m(\T^n,E)\hookrightarrow C^m(\T^n,E)\hookrightarrow B_{\infty\infty}^m(\T^n,E),
  \end{gather*}
  where ``$\hookrightarrow$'' means continuous embedding. Moreover, for all $s\in \R$ and $p,q\in[1,\infty]$, we have $C^\infty(\T^n,E)\hookrightarrow B_{pq}^s(\T^n,E)$.
  For $p,q,q_0,q_1\in[1,\infty],\, s_0,s_1\in\R$, and $\theta\in(0,1)$, the equality
  \[ \big( B_{pq_0}^{s_0}(\T^n,E), B_{pq_1}^{s_1}(\T^n,E)\big)_{\theta,q} = B_{pq}^{(1-\theta)s_0+\theta s_1}(\T^n,E)\]
  holds with equivalent norms, where $(\cdot,\cdot)_{\theta,q}$ denotes the real interpolation functor.
\end{remark}

In the following, we will discuss convolution in $\R^n$ and $\T^n$ and the connection to pseudodifferential operators in $\R^n$. For this, let us denote the Fourier transform in $\R^n$ by $\FR $ which is defined for $f\in L^1(\R^n,E)$ by
\[ (\FR f)(\xi) := (2\pi)^{-n/2} \int_{\R^n} e^{-ix\cdot\xi} f(x)dx\quad (\xi\in\R^n).\]
The inverse continuous Fourier transform is then given by $(\FR^{-1}g)(x) = (2\pi)^{-n/2} \int_{\R^n} e^{ix\cdot \xi} g(\xi)d\xi\;(x\in\R^n)$.

We will use the following result from \cite{barraza-denk-monzon14}, \textmd{Lemma~2.3.}

\begin{lemma}
  \label{3.9}
  Let $m\in\R$ and $a\in S_{1,0}^{m,n+1}(\R^n,L(E))$. Then $\FR^{-1}(\phi_j a)\in L^1(\R^n,L(E))$ and
  \[ \big\| \FR^{-1}(\phi_j a)\big\|_{L^1(\R^n,L(E))} \le c_{n,m} 2^{jm} \|a\|_{S_{1,0}^{m,n+1}(\R^n,L(E))}\quad (j\in\N_0),\]
  where the constants $c_{n,m}$ are independent of $j$ and $a$.
\end{lemma}

The convolution $f\ast g$ of $f\in L^1(\R^n,L(E))$ and $g\in L^p(\T^n,E)$ is defined by
\[ (f\ast g)(x) := \int_{\R^n} f(x-y)g(y)dy\quad (x\in \T^n),\]
identifying $g$ with its $2\pi$-periodic extension on $\R^n$. Note that $f\ast g$ is again $2\pi$-periodic and is considered as a function on $\T^n$ again. Similarly, for $f\in L^1(\T^n,L(E))$ and $g\in L^1(\T^n,E)$, we define the toroidal convolution by
\[ (f\ast g)_{\T^n}(x) := \int_{\T^n} f(x-y) g(y)\dbar y\quad (x\in\T^n).\]

\begin{lemma}
  \label{3.10}
  a) For $p\in[1,\infty]$, $f\in L^1(\R^n,L(E))$, and $g\in L^p(\T^n,E)$, we have $f\ast g\in L^p(\T^n,E)$ and
  \[ \|f\ast g\|_{L^p(\textmd{\T^n},E)} \le \|f\|_{L^1(\R^n,L(E))} \|g\|_{L^p(\T^n,E)}.\]

  b) Let $p\in[1,\infty]$, $f\in L^1(\T^n,L(E))$, and $g\in L^p(\T^n,E)$. Then we have $(f\ast g)_{\T^n}\in L^p(\T^n,E)$ and
  \[ \|(f\ast g)_{\T^n}\|_{L^p(\T^n,E)} \le \|f\|_{L^1(\T^n,L(E))} \|g\|_{L^p(\T^n,E)}.\]
\end{lemma}

\begin{proof}
  We only prove part a) as b) follows in the same way. As the cases $p=1$ and $p=\infty$ are straightforward, let us assume that $p\in (1,\infty)$ and set $q=p/(p-1)$. For $f\in L^1(\R^n,L(E))$ and $g\in L^p(\T^n,E)$ we can estimate, using H{\"o}lder's inequality,
  \begin{align*}
    \|f\ast g\|_{L^p(\T^n,E)}^p & \le \int_{\T^n}\textmd{\Big[} \int_{\R^n}  \|f(x-y)\|_{L(E)}\|g(y)\| dy\textmd{\Big]^p}\dbar x\\
    & \le \int_{\T^n} \Bigg[ \Big(\int_{\R^n} \|f(x-y)\|_{L(E)}dy\Big)^{1/q} \\
    & \qquad \cdot \Big( \int_{\R^n} \|f(x-y)\|_{L(E)} \|g(y)\|^pdy\Big)^{1/p}\Bigg]^p \dbar x\\
    & = \|f\|_{L^1(\R^n,L(E))}^{p/q} \int_{\T^n}\int_{\R^n} \|f(y)\|_{L(E)} \|g(x-y)\|^pdy \dbar x\\
    & = \textmd{\|f\|^{p/q+1} _{L^1(\R^n,L(E))}} \|g\|_{L^p(\T^n,E)}^p,
  \end{align*}
  where in the last step we used the invariance of the integral over $\T^n$ under translations. Therefore,
  \[ \|f\ast g\|_{L^p(\T^n,E)} \le \|f\|_{L^1 (\R^n,L(E))}^{1/p+1/q} \|g\|_{L^p(\T^n,E)},\]
  which yields a).
\end{proof}

\begin{lemma}
  \label{3.11}
If $M\in C_c(\R^n,L(E))\cap \FR(L^1(\R^n,L(E)))$, then for all $f\in C^\infty(\T^n,E)$ and $p\in[1,\infty]$ we have
\begin{align*}
  \FT^{-1}\big( M|_{\Z^n}\FT f\big)  & = (2\pi)^{-n/2}   (\FR^{-1}M)\ast f ,\\
  \Big\| \sum_{\bk \in\Z^n} \textmd{e_\bk}\otimes M(\bk ) & (\FT f)(\bk )\Big\|_{L^p(\T^n,E)}\\
   & \le \textmd{(2\pi)^{-n/2}}  \big\| \FR^{-1} M\big\|_{L^1(\R^n,L(E))} \|f\|_{L^p(\T^n,E)}.
\end{align*}
\end{lemma}

\begin{proof}
  Let $M$ be as in the lemma and $f\in C^\infty(\T^n,E)$, $p\in [1,\infty]$. In exactly the same way as in the case $n=1$ (see \cite{arendt-bu04}, Proposition~2.2), one obtains that for all $x\in\T^n$
  \begin{align*}
    \FT^{-1} (M|_{\Z^n} \FT f)(x) & = \Big[ \sum_{\bk\in\Z^n} e_{\bk}\otimes M(\bk)(\FT f)(\bk)\Big](x) \\
    & = (2\pi)^{-n/2} \big[ (\FR^{-1} M)\ast f\big](x).
  \end{align*}
  Now Lemma~\ref{3.10} a) yields the \textmd{second } statement.
\end{proof}

The following result shows that symbols on $\Z^n$ are restrictions of symbols on $\R^n$. It can be found in \cite{ruzhansky-turunen10}, Lemma~\textmd{II.4.5.1 } and Theorem~\textmd{II.4.5.3 } for the scalar-valued case, with the proofs carrying over to the vector-valued case.

\begin{proposition}
  \label{3.12}
  a) There exist $\theta\in\mathscr S(\R^n)$ and $\textmd{\theta_\alpha}\in\mathscr S(\R^n)$, $\alpha\in\N_0$, \textmd{such that}
  \[ (\mathcal P\theta)(x) := \sum_{\bk \in\Z^n} \theta (x+2\pi \bk) = 1\quad (x\in\R^n),\]
  $(\FR\theta)(\bk) = \delta_{0,\bk}\;(\bk\in\Z^n)$ and $\partial_\xi^\alpha (\FR\theta)(\xi) = \bar\Delta_{\xi}^\alpha\theta_\alpha(\xi)\;(\xi\in\R^n,\, \alpha\in\N_0^n)$.

  b) For $a\in S^{m,\rho}$ and $\theta$ as in a), define $\tilde a\colon\R^n\to L(E)$ by
  \begin{equation}
    \label{3-1}
    \tilde a(\xi) := \sum_{\bk\in\Z^n} (\FR\theta)(\xi-\bk) a(\bk)\quad (\xi\in\R^n).
  \end{equation}
  Then $\tilde a|_{\Z^n}= a$, and $\tilde a\in S^{m,\rho}_{1,0}(\R^n,L(E))$ with
  \[ \|\tilde a\|_{S^{m,\rho}_{1,0}(\R^n,L(E))} \le C_\theta \|a\|_{S^{m,\rho}}.\]
\end{proposition}

\begin{remark}
  \label{3.13}
  Writing the right-hand side of \eqref{3-1} in the form
  \[\sum_{\bk\in\Z^n} \textmd{\langle\bk\rangle^{2N} (\FR\theta)(\xi-\bk)\langle\bk\rangle^{-2N} a(\bk)}\]
  \textmd{for $N$ large enough, similarly to} Remark~\ref{3.4} a), we see that
  \[    \tilde a(\xi)  = \mathop{\mbox{$\displaystyle\mathrm{Os}\!-\!\!\!\!\sum\limits_{\bk\in\Z^n}\int_{\textmd{\R^n}}$}} e^{i(\bk-\xi)\cdot x} \textmd{\theta (x) a(\bk) d x } = \int_{\R^n} e^{-i\xi\cdot x} \textmd{\theta (x)(\FT^{-1}a)(x)}\textmd{d x}.\]
\end{remark}

\begin{theorem}
  \label{3.14}
  Let $p\in[1,\infty]$, $q\in[1,\infty)$ and $s\in\R$. Then $C^\infty(\T^n,E)$ is dense in $B_{pq}^s(\textmd{\T^n},E)$.
\end{theorem}

\begin{proof}
  We fix a second dyadic decomposition $(\psi_{\textmd{k}})_{\textmd{k}\in\N_0}$. For $u\in B_{pq}^s(\R^n,E)$ we set $u_N:=\sum_{k=0}^N \textmd{\op[\psi_k ]} u\;(N\in\N)$.   By definition of $B_{pq}^s(\T^n,E)$, we have $\textmd{\op[\psi_k ]}u\in L^p(\T^n,E)$ and therefore $u_N\in L^p(\T^n,E)\;(N\in\N)$. As $\FT u\in \mathcal O(\Z^n,E)$ and $\psi_k|_{\Z^n} \in \mathscr S(\Z^n)$, we see that $\psi_k|_{\Z^n}\FT u\in \mathscr S(\Z^n,E)$ and consequently $\textmd{\op[\psi_k ]}u\in C^\infty(\T^n, E)$.

  By Lemma~\ref{3.11},
  \[ \op[\phi_j ]\big( \op[\psi_k ]u\big) = (2\pi)^{-n/2} \big( \textmd{\FR^{-1}}\textmd{\psi_k}\big)\ast \big(\textmd{\op[\phi_j ]}u\big)\quad (j,k\in\textmd{\N_0}).\]
  Thus, $ \op[\phi_j ]\big( \op[\psi_k ]u\big)$ is a regular toroidal distribution and
  \begin{align*}
    \big\| \op[\phi_j ] & \big( \op[\psi_k ]u\big)\big\|_{L^p(\T^n,E)} \\
    &  \le (2\pi)^{-n/2}
    \|\FR^{-1}\psi_k\|_{L^1(\R^n)}
    \|\op[\phi_j ] u\|_{L^p(\T^n,E)}.
    \end{align*}
 \textmd{ We may assume that $(\psi_k)_{k\in\N_0}$ is constructed as described in Remark~\ref{3.5a}. In this case,
   \[\|\FR^{-1}\psi_k\|_{L^1(\R^n)} \le 2 \|\FR^{-1} \psi_0\|_{L^1(\R^n)}\quad (k\in\N_0).\]
   This yields
   \begin{equation}
     \label{3-2}
     \big\| \op[\phi_j ]  \big( \op[\psi_k ]u\big)\big\|_{L^p(\T^n,E)}\le c_{\psi_0} \|\op[\phi_j ] u\|_{L^p(\T^n,E)}.
   \end{equation}
}From Lemma~\ref{3.6} and $\supp\phi_j, \supp\psi_j\subset\{ \xi\in\R^n: 2^{j-1} \le|\xi|\le 2^{j+1}\}$, we obtain
  \begin{align}
    \|u-&u_N\|_{B_{pq}^s(\T^n,E)}^q = \sum_{j\in\N_0} 2^{jsq} \Big\| \op[\phi_j ] \Big( u-\sum_{k=0}^N \op[\psi_k ] u\Big)\Big\|_{L^p(\T^n,E)}^q \nonumber\\
    & = \sum_{j\in\N_0} 2^{jsq} \Big\| \sum_{k=N+1}^\infty \op[\phi_j ]\op[\psi_k ] u\Big\|_{L^p(\T^n,E)}^q \nonumber\\
    & = 2^{Nsq} \big\|\op[\phi_N ]\op[\psi_{N+1} ] u\big\|_{L^p(\T^n,E)}^{\textmd{q}}\nonumber \\
    & \quad + 2^{(N+1)sq} \big\| \op[\phi_{N+1} ]\big( \op[\psi_N ] + \op[\psi_{N+1} ]\big) u\big\|_{L^p(\T^n,E)}^q \nonumber\\
    & \quad + \sum_{j=N+2}^\infty 2^{jsq} \big\|\op[\phi_j ] u\big\|_{L^p(\T^n,E)}^q,
    \label{3-3}
  \end{align}
  where we used $\phi_j(\psi_{j-1} + \psi_j + \psi_{j+1}) = \phi_j$, \textmd{setting $\psi_{-1}:=0$}. We apply \eqref{3-2} and estimate the first and the second term in \eqref{3-3} by
  \begin{align*}
   2^{Nsq} \big\| \op[\phi_N]\op[\psi_{N+1}]u\big\|_{L^p(\T^n,E)}^q  &\le c_{\psi_0}^q 2^{Nsq} \|\op[\phi_N]u\|_{L^p(\T^n,E)}^q,\\
   2^{(N+1)sq} \big\|\op[\phi_{N+1}](\op[\psi_N]&+\op[\psi_{N+1}])u\big\|_{L^p(\T^n,E)} \\
    &\le 2^q c_{\psi_0}^q 2^{(N+1)sq} \|\op[\phi_{N+1}] u\|_{L^p(\T^n,E)}^q.
  \end{align*}
  Therefore,
  \[ \|u-u_N\|_{B_{pq}^s(\T^n,E)}^q \le c_{\psi_0,q} \sum_{j=N}^\infty \textmd{2^{jsq}} \|\op[\phi_j]u\|_{L^p(\T^n,E)}^q \to 0\;(N\to\infty).\]
   As $u_N\in C^\infty(\T^n,E)$, this proves the statement of the theorem.
\end{proof}

In the following, we will write $B_{pq}^{s,\infty}(\T^n,E)$ with $s\in\R$, $p,q\in[1,\infty]$ for $C^\infty(\T^n,E)$ endowed with the topology induced by $B_{pq}^s(\T^n,E)$. Note that the last result states that $B_{pq}^{s,\infty}(\T^n,E)$ is dense in $B_{pq}^s(\T^n,E)$ if $q<\infty$.

The following estimate is the essential step in the proof of the continuity of toroidal pseudodifferential operators in Besov spaces.

\begin{theorem}
  \label{3.15}
  Let $s,m\in\R$, $p,q\in[1,\infty]$, $\rho\in\N$ with $\rho\ge n+1$, and $a\in S^{m,\rho}$. Then
  \[ \op[a]\colon B_{pq}^{s+m,\infty}(\T^n,E) \to B_{pq}^s(\T^n,E)\]
  is linear and continuous with
    \[ \|\op[a]u\|_{B_{pq}^s(\T^n,E)} \le C \|a\|_{S^{m,\rho}} \|u\|_{B_{pq}^{s+m}(\T^n,E)}\quad (u\in B_{pq}^{s,\infty}(\T^n,E))\]
  with a constant $C$ not depending on $a$ or $u$.
\end{theorem}

\begin{proof}
  We only consider the case $q<\infty$, as the case $q=\infty$ follows similarly. Let $u\in B_{pq}^{s+m,\infty}(\T^n,E)$. We write \[ \|\op[a] u\|_{B_{pq}^s(\T^n,E)}^q = \sum_{j\in \N_0} 2^{jsq} \big\| \op\big[ \phi_j|_{\Z^n} a\big] u\big\|_{L^p(\T^n,E)}^q.\]
  We set $\chi_j := \phi_{j-1}+\phi_j+\phi_{j+1}$ (then $\chi_j\phi_j=\phi_j$) and extend $a$ to a symbol $\tilde a\in S^{m,\rho}_{1,0}(\R^n,L(E))$ as in Proposition~\ref{3.12}. Lemma~\ref{3.11} yields
  \begin{align*}
    \big\|\op\big[\phi_j|_{\Z^n}a\big]u&\big\|_{L^p(\T^n,E)} = \Big\| \sum_{\bk\in\Z^n} e_{\bk}\otimes \phi_j(\bk) a(\bk) \FT u(\bk)\Big\|_{L^p(\T^n,E)} \\
    & = \Big\| \sum_{\bk\in\Z^n} e_\bk \otimes \phi_j(\bk) \chi_j(\bk) \tilde a(\bk) \FT u(\bk)\Big\|_{L^p(\T^n,E)}\\
    & \le \big\| \FR^{-1} (\phi_j\tilde a)\big\|_{L^1(\R^n,L(E))} \|\op[\chi_j]u\|_{L^p(\T^n,E)}.
  \end{align*}
  Now we make use of the fact that
  \[ \big\| \FR^{-1} (\phi_j\tilde a)\big\|_{L^1(\R^n,L(E))} \le c_{n,m} 2^{jm}\|\tilde a\|_{S^{m,n+1}_{1,0}(\R^n,L(E))},\]
  \textmd{see Lemma~\ref{3.9}. } This gives, with \textmd{Proposition~\ref{3.12},}
  \begin{align*}
    \big\| \op\big[ \phi_j|_{\Z^n} a\big] u\big\|_{L^p(\T^n,E)} & \le c_{n,m,\phi} 2^{jm} \|\tilde a\|_{S^{m,n+1}_{1,0}(\R^n,L(E))} \|\op[\chi_j]u\|_{L^p(\T^n,E)} \\
    & \le c_{n,m,\phi} \textmd{C_\theta} \|a\|_{S^{m,\rho}} 2^{jm} \|\op[\chi_j]u\|_{L^p(\T^n,E)}.
  \end{align*}
  Summing up over $j$, we get
  \begin{align*}
   \|\op[a]u\|_{B_{pq}^s(\T^n,E)}^q &\le (1+2^{sq} + 2^{-sq}) c_{n,m,\phi}^q \textmd{C_\theta^q} \|a\|_{S^{m,\rho}}^q \\
    &\cdot \sum_{j\in\N_0} 2^{j(s+m)q} \|\op[\phi_j]u\|_{L^p(\T^n,E)}^q.
  \end{align*}

  \vspace*{-1em}
\end{proof}

Now we are able to prove one of the main results of the present paper.

\begin{theorem}
  \label{3.16}
  Let $m,s\in\R$, $p,q\in[1,\infty]$, $\rho\in\N$ with $\rho\ge n+1$, and $a\in S^{m,\rho}$. Then
  \[ \op[a]\colon B_{pq}^{s+m}(\T^n,E)\to B_{pq}^s(\T^n,E)\]
  is linear and continuous. \textmd{Moreover,
   \[ \big( a\mapsto \op[a]\big) \in L\Big( S^{m,\rho}, L(B_{pq}^{s+m}(\T^n,E), B_{pq}^s(\T^n,E))\Big).\]}
\end{theorem}

\begin{proof}
  For $q<\infty$, the statement follows immediately from Theorem~\ref{3.15} and the density of $B_{pq}^{s+m,\infty}(\T^n,E)$ in $B_{pq}^{s+m}(\T^n,E)$, Theorem~\ref{3.14}.

  The case $q=\infty$ will be treated with real interpolation theory (see, e.g., \cite{lunardi95} for a survey on interpolation theory). In fact, we have
  \[ B_{p\infty}^{r}(\T^n,E) = \Big( B_{p1}^{\textmd{r-1}}(\T^n,E), B_{p1}^{\textmd{r+1}}(\T^n,E)\Big)_{1/2,\infty}\]
  for $r\in \{s,s+m\}$. Now the continuity of
  \[ \op[a]\colon B_{p1}^{s+m\pm1}(\T^n,E)\to B_{p1}^{s\pm1}(\T^n,E)\]
  and the properties of the real interpolation functor immediately give the continuity of
  \[ \op[a]\colon B_{p\infty}^{s+m}(\T^n,E)\to B_{p\infty}^s(\T^n,E).\]
  \textmd{In the same way, the continuity of the map $a\mapsto \op[a]$ follows.}
\end{proof}

\section{Generation of semigroups for parabolic pseudodifferential operators}

Now we will investigate the generation of analytic semigroups by realizations of toroidal pseudodifferential operators in Besov and Sobolev spaces. \textmd{For this, one of the key ingredients is an estimate of the inverse symbol $(a(\bk)+\lambda)^{-1}$ in the case of a parabolic symbol $a\in S^{m,\rho}$. Therefore, we start with a remark on the discrete derivatives.

\begin{remark}
  \label{4.0}
  Let $a\colon \Z^n\to L(E)$ and $\lambda\in\C$ such that $a(\bk)+\lambda$ is invertible for all $\bk\in\Z^n$. Then the discrete difference $\Delta_\bk^\gamma (a(\bk)+\lambda)^{-1}$ can be written as a finite linear combination of terms of the form
  \begin{equation}
   \label{4-0}
   (a\big(\bk+\beta_0^{(j)})+\lambda\big)^{-1} \prod_{l=1}^j \Big[ (\Delta_\bk^{\alpha^{(l)}} a)(\bk+\beta_1^{(l)}) \big( a(\bk+\beta_2^{(l)})+\lambda\big)^{-1}\Big],
   \end{equation}
  where $\alpha^{(l)}, \beta_i^{(l)}\in\N_0^n$ and $j\in\{1,\dots,|\gamma|\}$, and where $\sum_{l=1}^j |\alpha^{(l)}| = |\gamma|$.

  This statement follows by induction in a straightforward way, based on the discrete Leibniz formula
  \begin{equation}
    \label{4-0a}
\Delta_\bk^\alpha (fg)(\bk) = \sum_{\substack{\beta\in\N_0\\ \beta\le\alpha}} \binom \alpha\beta \big[ \Delta_\bk^\beta f\big](\bk) \big[\Delta_\bk^{\alpha-\beta} g\big](\bk+\beta)\quad (\bk\in\Z^n)
 \end{equation}
  for $f,g\colon\Z^n\to L(E)$, which can be found (in the scalar case) in \cite{ruzhansky-turunen10}, Lemma~II.3.3.6.
\end{remark}
}
\textmd{Throughout this section, we } fix $s\in\R$, $m\in (0,\infty)$, $p,q\in[1,\infty]$ and $\rho\in\N$ with $\rho\ge n+1$. For $a\in S^{m,\rho}$, we denote its $B_{pq}^{s}(\T^n,E)$-realization by $\opp a$, i.e. we define  $\opp a$ as an unbounded operator  in \textmd{$B_{pq}^s(\T^n,E)$ } with domain $D(\opp a) := B_{pq}^{s+m}(\T^n,E)$ acting as $\opp a u := \op[a]u\;(u\in D(\opp a))$.

\begin{definition}
  \label{4.1}
  The symbol $a\in S^{m,\rho}$ is called parabolic with constants $\omega\ge 0$ and $\kappa>0$ if for all \textmd{$(\bk,\lambda)\in\Z^n\times \C$ with $\Re\lambda\ge 0$ } and $|(\bk,\lambda)|\ge \omega$ we have that $a(\bk)+\lambda\colon E\to E$ is bijective and
  \[ \| (a(\bk)+\textmd{\lambda})^{-1} \|_{L(E)} \le \kappa \langle\bk,\textmd{\lambda}\rangle^{-m}.\]
  Here, \textmd{$\langle \bk,\lambda\rangle := (1+|\bk|^2+|\lambda|^{2/m})^{1/2}$ and $|(\bk,\lambda)| := (|\bk|^2+|\lambda|^{2/m})^{1/2}$.}
  We write $P_{\omega,\kappa}S^{m,\rho}$ for the set of all symbols in $S^{m,\rho}$ which are parabolic with constants $\omega$ and $\kappa$ and endow this set with the topology induced by the topology in $S^{m,\rho}$.
\end{definition}

In the following, for $R\ge 0$ and $\theta\in (0,\pi]$ let
\[ \Sigma_{\theta,R} := \{\lambda\in\C: |\lambda|\ge R, \, |\arg(\lambda)|\le \theta\}.\]

\begin{theorem}
  \label{4.2} Let $\omega\ge 0$ and $\kappa>0$, and let $\mathcal A\subset P_{\omega,\kappa}S^{m,\rho}$ be bounded. Then, for each $a\in \mathcal A$ its $B_{pq}^{s}(\T^n,E)$-realization $\opp a$ satisfies $\Sigma_{\pi/2,\omega^m}\subset\rho(-\opp a)$ and
  \[ (1+|\lambda|)^{1-j} \|(\opp a+\lambda)^{-1} \|_{L(B_{pq}^s(\T^n,E), B_{pq}^{s+jm}(\T^n,E))} \le C\quad (\lambda\in \Sigma_{\pi/2,\omega^m},\, j=0,1),\]
  where $C$ is independent of $\lambda$ and $a$. Furthermore, for all $a\in\mathcal A$ and $\lambda\in \Sigma_{\pi/2,\omega^m}$ we have $(\opp a+\lambda)^{-1} = \textmd{\op[b_\lambda]}\big|_{B_{pq}^s(\T^n,E)}$ where $b_\lambda(\bk) := (a(\bk)+\lambda)^{-1}\;(\bk\in\Z^n)$.
\end{theorem}

\begin{proof}
  As the proof is similar to the proof of the analog result in the continuous case in \cite{amann97}, Theorem~7.2, we only indicate the main steps. Using the parabolicity and continuity assumptions, it is straightforward to see that
  \begin{align*}
    \langle\bk\rangle^m \|(a(\bk)+\lambda)^{-1}\|_{L(E)} & \le\kappa,\\
    (1+|\lambda|) \|(a(\bk)+\lambda)^{-1} \|_{L(E)} & \le 2\kappa
  \end{align*}
  holds for all $\bk\in \Z^n$, $\lambda\in \Sigma_{\pi/2,\omega^m}$ and $a\in \mathcal A$. \textmd{To deal with the discrete derivatives, we use the description \eqref{4-0} from Remark~\ref{4.0} and obtain the estimate}
  \[ \langle \bk\rangle^{|\alpha|+jm} (1+|\lambda|)^{1-j} \|\Delta_\bk^\alpha (a(\bk)+\lambda)^{-1} \|_{L(E)} \le C\]
  for all $|\alpha|\le \rho$, $\bk\in\Z^n$, $\lambda\in\Sigma_{\pi/2,\omega^m}$, $a\in\mathcal A$, and $j=0,1$. Therefore, $b_\lambda = (a+\lambda)^{-1} \in S^{-m,\rho}$ and
  \[ (1+|\lambda|) \|(a+\lambda)^{-1}\|_{S^{0,\rho}} \le C\quad (\lambda\in \Sigma_{\pi/2,\omega^m},\, a\in\mathcal A).\]
  Now \textmd{Theorem~\ref{3.16} } implies that
  \[ (1+|\lambda|)^{1-j} \|\textmd{\op[b_\lambda}]\|_{L(B_{pq}^s(\T^n,E), B_{pq}^{s+jm}(\T^n,E))} \le C.\]
  As we also have $\textmd{\op[a]}+\lambda\in L(B_{pq}^{s+m}(\T^n,E), B_{pq}^s(\T^n,E))$ by Theorem~\ref{3.16} and due to
  \begin{align*}
    (\textmd{\op[a]}+\lambda) b_\lambda(D) & = \id_{B_{pq}^s(\T^n,E)},\\
    \textmd{\op[b_\lambda]} (\textmd{\op[a]}+\lambda) & = \id_{B_{pq}^{s+m}(\T^n,E)}
  \end{align*}
  by Remark ~\ref{3.4} b), we see that $\Sigma_{\pi/2,\omega^m}\subset\rho(-\opp a)$ and $(\opp a+\lambda)^{-1} =\textmd{\op[b_\lambda]}|_{B_{pq}^s(\T^n,E)}$.
\end{proof}

\begin{corollary}
  \label{4.3}
  Let $a\in S^{m,\rho}$ be parabolic, and let $\opp a$ be its $B_{pq}^s(\T^n,E)$-realization. Then $-\opp a$ generates an analytic semigroup in $B_{pq}^s(\T^n,E)$.
\end{corollary}

\begin{proof}
  We know from Theorem~\ref{4.2} that $\Sigma_{\pi/2,R}\subset\rho(-\opp a)$ and
  \[ \sup_{\lambda\in\Sigma_{\pi/2,R}} (1+|\lambda|) \|(\opp a+\lambda)^{-1}\|_{L(E)} <\infty\]
  holds for sufficiently large $R>0$. Due to the fact that the set of all angles where the parabolicity conditions hold is open, we can replace $\Sigma_{\pi/2,R}$ by $\Sigma_{\theta,R}$ with some $\theta\in (\frac\pi 2,\pi]$. Now the statement follows from standard semigroup theory (see, e.g., \cite{lunardi95}, Chapter~2).
\end{proof}

Whereas in the results above, the generation of semigroups in Besov spaces could be shown quite easily, the same question in the context of Sobolev spaces $W_p^k(\T^n,E)$ is more difficult to answer. \textmd{We start with some preliminary remarks. }

Throughout the following, we fix $\omega\ge0$, $\kappa>0$ and a bounded subset $\mathcal A\subset P_{\omega,\kappa}S^{m,\rho}$. \textmd{We also fix a positive $R$ with $R\ge\omega^m$. }
 We start with a preliminary estimate, where in the following $C$ stands for a generic constant which may vary from one appearance to another.

\begin{lemma}
  \label{4.4}
  Let $\theta_0 := \frac12\min\{m,1\}$ and $\theta_1:=1-\theta_0$. For all $\mu\ge \mu_0>0$ we have
  \begin{align}
    \sum_{\ell\in\mu^{-1}\Z^n} |\ell|^{\theta_0} \textmd{(1+\mu^2+|\mu\ell|^2)^{-m}} \langle\mu\ell\rangle^{m-n}  & \le C \mu^{-m}\label{4-1},\\
     \sum_{\ell\in\mu^{-1}\Z^n} |\ell|^{\theta_1} \textmd{(1+\mu^2+|\mu\ell|^2)^{-m}} \langle \mu\ell\rangle^{m-n-1} & \le C\mu^{-m-1}\label{4-2}.
\end{align}
\end{lemma}

\begin{proof}
  \textbf{(i)} We will distinguish the cases $|\ell|_\infty\le 1$ and $|\ell|_\infty>1$ and first show that
  \begin{align}
    \sum_{\substack{\ell\in\mu^{-1}\Z^n\setminus \{0\} \\ |\ell|_\infty\le 1}} |\ell|^{\theta-n} \mu^{-n} & \le C_\theta && (\theta>0) ,\label{4-3}\\
    \sum_{\substack{\ell\in\mu^{-1}\Z^n \\ |\ell|_\infty> 1}} |\ell|^{\theta-n-m}\mu^{-n} & \le C_\theta && (\theta\in (0,m)) \label{4-4}.
  \end{align}
  To prove \eqref{4-3}, we set $Y_N := \{ \bk\in\Z^n: |\bk|_\infty = N\},\, N\in\N$. Then $\Z^n\setminus\{0\}=\bigcup_{N\in\N} Y_N$ and $\card Y_N\le C N^{n-1}$. Therefore,
  \begin{align*}
    \sum_{\substack{\ell\in\mu^{-1}\Z^n\setminus \{0\}\\|\ell|_\infty\le 1}} &|\ell|^{\theta-n} \mu^{-n} =
    \sum_{\substack{\bk\in\Z^n\setminus \{0\}\\ |\bk|_\infty\le \mu}} |\bk|^{\theta-n} \mu^{-\theta}\le \mu^{-\theta} \sum_{N=1}^{[\mu]+1}\sum_{\bk\in Y_N} |\bk|^{\theta-n} \\
    & \le C \mu^{-\theta} \sum_{N=1}^{[\mu]+1} N^{n-1} N^{\theta-n} \le C\mu^{-\theta} \int_0^{[\mu]} x^{\theta-1} dx \le C_\theta<\infty.
  \end{align*}

  For the proof of \eqref{4-4}, we define $Z_N := \{\ell\in\mu^{-1}\Z^n: N-1<|\ell|_\infty\le N\}$. Then $\card Z_N \le C \mu^n N^{n-1}$, and we get
  \begin{align*}
    \sum_{\substack{\ell\in\mu^{-1}\Z^n \\ |\ell|_\infty> 1}} & |\ell|^{\theta-n-m}\mu^{-n} =  \sum_{N=2}^\infty \sum_{\ell\in Z_N} |\ell|^{\theta-n-m}\mu^{-n} \\
    & \le C \sum_{N=2}^\infty N^{n-1} (N-1)^{\theta-n-m} \le  C \sum_{N=1}^\infty N^{\theta-m-1} =C_\theta<\infty.
  \end{align*}

  \textbf{(ii)} Now we prove \eqref{4-1}. We use the inequality
  \begin{equation}
     \label{4-5}
     \textmd{(1+\mu^2+|\mu\ell|^2)^{-m}} \le \mu^{-m} \langle \ell\rangle^{-m} \langle \mu\ell\rangle^{-m}.
   \end{equation}
  For the sum over all $\ell\in\mu^{-1}\Z^n$ with $|\ell|_\infty\le 1$, we obtain with \eqref{4-3},  with $\langle \ell\rangle\approx C$ and with $\langle\mu\ell\rangle\ge \mu|\ell|$
  \begin{align*}
    \mu^{m}\sum_{\substack{\ell\in\mu^{-1}\Z^n\\ |\ell|_\infty\le 1}} |\ell|^{\theta_0} & \textmd{(1+\mu^2+|\mu\ell|^2)^{-m}} \langle\mu\ell\rangle^{m-n} \le \sum_{\substack{\ell\in\mu^{-1}\Z^n\\ |\ell|_\infty\le 1}} |\ell|^{\theta_0}
        \langle\mu\ell\rangle^{-n}\\
        &\le \sum_{\substack{\ell\in\mu^{-1}\Z^n\textmd{\setminus\{0\}}\\ |\ell|_\infty\le 1}} |\ell|^{\theta_0-n} \mu^{-n}\le C_{\theta_0}.
  \end{align*}
  For $\ell\in\mu^{-1}\Z^n$ with $|\ell|_\infty>1$, we use  \eqref{4-5} to see that
  \begin{align*}
     \mu^m\sum_{\substack{\ell\in\mu^{-1}\Z^n\\ |\ell|_\infty>1}} & |\ell|^{\theta_0} \textmd{(1+\mu^2+|\mu\ell|^2)^{-m}} \langle\mu\ell\rangle^{m-n} \le \sum_{\substack{\ell\in\mu^{-1}\Z^n\\ |\ell|_\infty>1}}|\ell|^{\theta_0-n-m}\mu^{-n} \le C_{\theta_0},
  \end{align*}
  where the last inequality follows from \eqref{4-4}.

    \textbf{(iii)} Finally, we show that \eqref{4-2} holds. For $\ell\in\mu^{-1}\Z^n$ with $|\ell|_\infty\le 1$, we write
  \begin{align}
   \mu^{m+1}\sum_{\substack{\ell\in\mu^{-1}\Z^n\\ |\ell|_\infty\le 1}} & |\ell|^{\theta_1} \textmd{(1+\mu^2+|\mu\ell|^2)^{-m}} \langle \mu\ell\rangle^{m-n-1} \nonumber\\
   & \le\sum_{\substack{\ell\in\mu^{-1}\Z^n\\ |\ell|_\infty\le 1}} \mu^{-m+1} |\ell|^{\theta_1} \langle\mu\ell\rangle^{m-n-1}\nonumber \\
   & \le  \sum_{\substack{\ell\in\mu^{-1}\Z^n\setminus\{0\}\\ |\ell|_\infty\le 1}} |\ell|^{\theta_1-n}\mu^{-n} \Big[\mu^{1-m} \langle\mu\ell\rangle^{m-1}\Big].\label{4-6}
   \end{align}
   If $m\ge 1$, then $\langle\mu\ell\rangle^{m-1} \le   \langle \mu\rangle^{m-1}\le C \mu^{m-1}$ due to $|\ell|_\infty\le 1$, with $C$ depending on $\mu_0$. Therefore, the bracket $[\ldots]$ in \eqref{4-6} is bounded by a constant, and consequently the sum in \eqref{4-6} is bounded by a constant, too, due to \eqref{4-3} with $\theta:=\theta_1$.

   If $m< 1$, we estimate $\langle\mu\ell\rangle\ge \mu|\ell|$ and see that the sum in \eqref{4-6} is not greater than
   \[  \sum_{\substack{\ell\in\mu^{-1}\Z^n\setminus\{0\}\\ |\ell|_\infty\le 1}} \ell^{\theta_1+m-1-n}\mu^{-n} \le C_{\theta_1},\]
    applying \eqref{4-3} with $\theta:= \theta_1+m-1 =m-\theta_0 >0$. In both cases we see that the sum in \eqref{4-6} is bounded by a constant.

  We still have to estimate the sum \eqref{4-2} over all $\ell\in\mu^{-1}\Z^n$ with $|\ell|_\infty > 1$. If $m\ge 1$, we use $\mu\langle\mu\ell\rangle^{-1} \le 1$ and get
  \begin{align*}
  \mu^{m+1}\sum_{\substack{\ell\in\mu^{-1}\Z^n\\ |\ell|_\infty> 1}}  &|\ell|^{\theta_1} \textmd{(1+\mu^2+|\mu\ell|^2)^{-m}}  \langle \mu\ell\rangle^{m-n-1} \\
  & \le  \mu^m \sum_{\substack{\ell\in\mu^{-1}\Z^n\\ |\ell|_\infty> 1}}  |\ell|^{\theta_1} \textmd{(1+\mu^2+|\mu\ell|^2)^{-m}} \langle \mu\ell\rangle^{m-n} \le C_{\theta_1},
  \end{align*}
  where the last inequality was already shown in part (ii) of the proof.

  For $m<1$ we apply the inequalities $\langle\mu\ell\rangle^{m-n-1}\le \mu^{m-n-1}|\ell|^{m-n-1}$,  $ \langle \mu\ell,\mu\rangle^{-2m}\le \mu^{-2m}|\ell|^{-2m}$ and $ 1<|\ell|_\infty \le |\ell|$ to obtain
  \begin{align*}
  \mu^{m+1}\sum_{\substack{\ell\in\mu^{-1}\Z^n\\ |\ell|_\infty> 1}}  &|\ell|^{\theta_1} \langle\mu\ell,\mu\rangle^{-2m} \langle \mu\ell\rangle^{m-n-1} \\
  & \le  \sum_{\substack{\ell\in\mu^{-1}\Z^n\\ |\ell|_\infty> 1}}  |\ell|^{\theta_1-1-n-m}  \mu^{-n}  \\
  & \le  \sum_{\substack{\ell\in\mu^{-1}\Z^n\\ |\ell|_\infty> 1}} |\ell|^{(\theta_1+m-1)-n-m} \mu^{-n}
   \le C_{\theta_1},
  \end{align*}
  now using \eqref{4-4} with $\theta:= \theta_1+m-1 = m-\theta_0 \in (0,m)$.

\end{proof}

\begin{lemma}
  \label{4.5}
a) For all $\gamma\in\N_0^n\setminus\{0\}$ with $|\gamma|\le \rho$, $a\in\mathcal A$, and $\lambda\in\Sigma_{\pi/2,R}$ we have
  \begin{equation}
    \label{4-7}
    \textmd{\big\| \Delta_{\bk}^\gamma (a(\bk)+\lambda)^{-1}\big\|_{L(E)} \le C \langle\bk,\lambda\rangle^{-2m} \langle\bk\rangle^{m-|\gamma|}\quad (\bk\in\Z^n).}
  \end{equation}

  b) Let $\varphi\in\mathscr S(\R)$ with $\varphi(0)=1$. For $\epsilon\in (0,1)$ set $\varphi_\epsilon(x) := \varphi(\epsilon x)\;(x\in\R)$ and $\chi_\epsilon(\xi,\textmd{\lambda}) := \varphi_\epsilon ((|\xi|^2+\textmd{|\lambda|^{2/m}})^{1/2})\;(\xi\in\R^n,\,\textmd{\lambda\in \Sigma_{\pi/2,R}})$. Then for all $\gamma\in\N_0^n$ and $\textmd{\lambda\in\C\setminus\{0\}}$ we have
  \[ \textmd{\big| \Delta_{\bk} ^\gamma \chi_\epsilon(\bk,\lambda)\big| \le C \langle\bk, \lambda\rangle^{-|\gamma|}\quad (\bk\in \Z^n)}\]
  with a constant $C$ depending on $\gamma$ but not on $\epsilon$ or $\lambda$.

  c) For all $\gamma\in\N_0^n$ with $0<|\gamma|\le \rho$, $a\in\mathcal A$ and $\lambda\in\Sigma_{\pi/2,R}$ we have
  \[ \textmd{\big\| \Delta_{\bk}^\gamma \big[ \chi_\epsilon(\bk,\lambda)(a(\bk)+\lambda)^{-1}\big]\big\|_{L(E)} \le C \langle\bk,\lambda\rangle^{-2m} \langle\bk\rangle^{m-|\gamma|}\quad (\bk\in \Z^n).}\]
\end{lemma}

\begin{proof}
a)   \textmd{Again we use Remark~\ref{4.0} to write $\Delta_\bk^\gamma (a(\bk)+\lambda)^{-1}$ as a finite linear combination of products of the form \eqref{4-0}. Due to $|\gamma|>0$, a term of the form  $(a(\bk+\ldots)+\lambda)^{-1}$ } appears at least twice \textmd{in \eqref{4-0}, } and by the parabolicity condition and by $a\in S^{m,\rho}$ we obtain
  \[ \big\| \Delta_\bk^\gamma (a(\bk)+\lambda)^{-1} \big\|_{L(E)} \le C \langle \bk,\textmd{\lambda}\rangle^{-2m} \langle\bk\rangle^{m-|\gamma|}.\]

b) Let $g(\xi,\textmd{\lambda}):= (|\xi|^2+\textmd{|\lambda|^{2/m}})^{1/2}$ so that $\chi_\epsilon = \varphi_\epsilon\circ g$. By  \cite{kumanogo81}, Chapter~1, Lemma~6.3, we have for the $k$-th derivative of $\varphi_\epsilon$
\begin{equation}
  \label{4-8}
  |\varphi_\epsilon^{(k)}(x) \|\le c_k \langle x\rangle^{-k}\quad (k\in\N_0,\, x\in\R)
\end{equation}
with a constant $c_k$ independent of $\epsilon\in (0,1)$. On the other hand, $g$ is positively \textmd{quasi-}homogeneous of degree $1$ in $(\xi,\textmd{\lambda})$, i.e. $g(\rho\xi,\textmd{\rho^m\lambda}) = \rho g(\xi,\textmd{\lambda})$ holds for \textmd{$(\xi,\lambda)\in(\R^n\times\C)\setminus\{0\}$ } and $\rho>0$. Therefore, the $\alpha$-th derivative is positively \textmd{quasi-}homogeneous of degree $1-|\alpha|$, i.e.,
\[ \partial^\alpha_{\textmd{\xi}} g(\rho\xi,\textmd{\rho^m\lambda}) = \rho^{1-|\alpha|} (\partial^\alpha_{\textmd{\xi}} g)(\xi,\textmd{\lambda})\quad \textmd{(\xi,\lambda)\in(\R^n\times\C)\setminus\{0\}},\, \rho>0).\]
As $g$ is $C^\infty$ in $(\R^n\times\C)\setminus\{0\}$ and therefore bounded with all derivatives on the compact set $\{(\xi,\textmd{\lambda})\in\R^n\times \C: |\xi|^2+\textmd{|\lambda|^{2/m}}=1\}$, this implies, setting $\rho:= (|\xi|^2+\textmd{|\lambda|^{2/m}})^{1/2}$,
\begin{equation}
  \label{4-9}
  |\partial^\alpha_{\textmd{\xi}} g(\xi,\textmd{\lambda})| \le C_\alpha (|\xi|^2+\textmd{|\lambda|^{2/m}})^{(1-|\alpha|)/2}\quad \textmd{ ((\xi,\lambda)\in(\R^n\times\C)\setminus\{0\}).}
\end{equation}
To estimate $\partial^\gamma_{\textmd{\xi}}\chi_\epsilon$, we apply the generalized chain rule which can be formulated as follows (see, e.g., \cite{bahouri-chemin-danchin11}, Lemma~2.3): The $\gamma$-th derivative $[\partial^\gamma_{\textmd{\xi}} (\varphi_\epsilon\circ g)](\xi,\mu)$ is a finite linear combination of terms of the form
\textmd{\[ \varphi_\epsilon^{(k)}(g(\xi,\lambda))\cdot (\partial^{\alpha^{(1)}}_{\textmd{\xi}} g)(\xi,\lambda)\cdot\ldots\cdot  (\partial^{\alpha^{(k)}}_{\textmd{\xi}} g)(\xi,\lambda), \]
where $k\in \{1,\dots,|\gamma|\}$ and $\alpha^{(1)},\dots, \alpha^{(k)}\in\N_0^n$ with $|\alpha^{(1)}|+\ldots+|\alpha^{(k)}| = |\gamma|$.}

From this and the estimates \eqref{4-8} and \eqref{4-9}, we obtain
\[ |\partial^\gamma_{\textmd{\xi}} \chi_\epsilon(\xi,\textmd{\lambda})|\le C_\gamma (|\xi|^2+\textmd{|\lambda|^{2/m}})^{-|\gamma|/2}.\]
Here the constant $C_\gamma$ depends on $\gamma$ and $\chi$ but not on $\epsilon$, $\xi$ or $\mu$. Now we apply the mean value theorem (see \cite{ruzhansky-turunen10}, proof of Theorem~\textmd{II.4.5.3}),
\[ \Delta_\bk^\gamma \chi_\epsilon(\bk,\textmd{\lambda}) = \partial_\xi^\gamma \chi_\epsilon(\xi,\textmd{\lambda})\big|_{\xi=\tilde \xi}\]
for some $\tilde \xi \in [k_1,k_1+\gamma_1]\times \ldots\times [k_n, k_n+\gamma_n]$. Therefore,
\[ \big| \Delta _\bk^\gamma \chi_\epsilon(\bk,\textmd{\lambda}) \big| \le C (|\bk|^2+\textmd{|\lambda|^{2/m}})^{-|\gamma|}\quad (\bk \in\Z^n,\,\textmd{\lambda\in\C^n\setminus\{0\}}).\]

c) This follows from a) and b) by the Leibniz rule \textmd{\eqref{4-0a} } and the inequality $\langle\bk\rangle \langle\bk,\lambda\rangle^{-1}\le 1$.
\end{proof}

\textmd{\begin{remark}
  \label{4.6}
  In the proof of the following lemma, we will use the elementary inequality
  \begin{equation}
    \label{4-10}
    |\eta|^N \le C  \sum_{|\gamma|=N} \big| (e^{-i\eta}-1)^\gamma \big|\quad (\eta\in\T^n),
  \end{equation}
  for $N\in\N$, where we have set $(e^{-i\eta}-1)^\gamma := (e^{-i\eta_1} -1)^{\gamma_1}\cdot\ldots\cdot (e^{-i\eta_n}-1)^{\gamma_n}$.

  To prove \eqref{4-10}, we consider the function $f(\eta) := \sum_{|\gamma|=N} | (e^{-i\eta}-1)^\gamma |$. Obviously, $f$ is a continuous function on $\T^n$ and has no zeros for $|\eta| \ge\frac\pi 2$ (choose $\gamma_j := N$ for $|\eta_j| = \max\{ |\eta_1| ,\dots,|\eta_n|\}$). Therefore $f(\eta)\ge C >0\;(|\eta|\ge\frac\pi 2)$ which implies \eqref{4-10} for $|\eta|\ge \frac\pi 2$.

  For $|\eta|\le\frac\pi 2$, we use
  \[ |(e^{-i\eta}-1)^\gamma|  = \prod_{j=1}^n |e^{-i\eta_j}-1|^{\gamma_j} \ge \prod_{j=1}^n |\sin(\eta_j)|^{\gamma_j} \ge C\prod_{j=1}^n |\eta_j|^{\gamma_j} . \]
  Now \eqref{4-10} follows from
  \[   |\eta|^N \le C\sum_{|\gamma|=N} |\eta^\gamma| = C\sum_{|\gamma|=N} \prod_{j=1}^n |\eta_j|^{\gamma_j} = C\sum_{|\gamma|=N} \prod_{j=1}^n |\eta_j|^{\gamma_j}.\]
\end{remark}
}
The following result is the key estimate in the proof of the generation of an analytic semigroup in Sobolev spaces.

\begin{lemma}
  \label{4.7}
  With $\chi_\epsilon$ being defined as in Lemma~\ref{4.5}, define
  \[ K_\epsilon (\eta,\lambda) := \sum_{\bk\in\Z^n} e^{i\bk \cdot\eta} \chi_\epsilon(\bk,\textmd{\lambda})(a(\bk)-\lambda)^{-1}\]
  for $\epsilon\in(0,1)$, $\eta\in\T^n$ and $\lambda\in\Sigma_{\pi/2,R}$.
  Choose $\theta_0$ and $\theta_1$ as in Lemma~\ref{4.4}. Then
  \begin{equation}\label{4-11}
   |\lambda|\, \|K_\epsilon(\eta,\lambda)\|_{L(E)} \le C \, \frac{ \mu^{\theta_0} |\eta| ^{\theta_0} + \mu^{\theta_1} |\eta| ^{\theta_1}}{ |\eta| ^n (1+\mu|\eta| )}\quad (\eta\in\T^n,\, \lambda\in\Sigma_{\pi/2,R})
   \end{equation}
  with $C$ being independent of $\epsilon$ and $\lambda$, \textmd{where we have set $\mu:=|\lambda|^{1/m}$. Further, }
    \begin{equation}\label{4-12}
  |\lambda|\,\|K_\epsilon(\cdot,\lambda)\|_{L^1(\T^n,L(E))} \le C \quad \textmd{(\lambda\in\Sigma_{\pi/2,R}).}
  \end{equation}
  Moreover, there exists a strongly measurable function $K\colon \T^n\times \Sigma_{\pi/2,R} \to L(E)$ with $K_\epsilon(\eta,\lambda)\to K(\eta,\lambda)\;(\epsilon \searrow 0)$ pointwise almost everywhere, and \eqref{4-11} and \eqref{4-12} hold with $K_\epsilon$ being replaced by $K$.
\end{lemma}

\begin{proof}
   Let $\gamma\in\N_0^n$ with $|\gamma|=n+i$, $i\in\{0,1\}$.  \textmd{It is easy to see that
  \[ (e^{-i\gamma}-1) e^{i\bk\cdot \eta} = (-1)^{|\gamma|} {\bar\Delta}\,_\bk^\gamma e^{i\bk\cdot\eta}.\]
  We will also imply the summation by parts formula which states that for $f,g\colon\Z^n\to L(E)$ we have
  \[ \sum_{\bk\in\Z^n} f(\bk) \big[\Delta_{\bk}^\gamma g\big](\bk) = (-1)^{|\gamma|} \sum_{\bk\in\Z^n} \big[ {\bar\Delta}\,_\bk^\gamma f\big](\bk) g(\bk)\]
  (see \cite{ruzhansky-turunen10}, Lemma~II.3.3.10).}

  By this, we obtain
  \begin{align*}
    (e^{-i\eta}-1)^\gamma & K_\epsilon(\eta,\lambda) = (e^{\textmd{-i\eta}}-1)^\gamma \sum_{\bk\in\Z^n} e^{i\bk\cdot\eta} \chi_\epsilon(\bk,\textmd{\lambda}) (a(\bk)-\lambda)^{-1}\\
    & = \sum_{\bk\in\Z^n} (-1)^{|\gamma|} \big[ {\bar\Delta}\,_\bk^\gamma e^{i\bk\cdot\eta}\big] \chi_\epsilon(\bk,\textmd{\lambda}) (a(\bk)-\lambda)^{-1} \\
    & = \sum_{\bk\in\Z^n} e^{i\bk\cdot\eta} \Delta_\bk^\gamma\big[ \chi_\epsilon(\bk,\textmd{\lambda}) (a(\bk)-\lambda)^{-1} \big]\\
    & = \sum_{\bk\in\Z^n} \big( e^{i\bk\cdot\eta} -1\big)      \Delta_\bk^\gamma\big[ \chi_\epsilon(\bk,\textmd{\lambda}) (a(\bk)-\lambda)^{-1} \big].
      \end{align*}
    \textmd{In the last step } we used the equality
    \[ \sum_{\bk\in\Z^n}  \Delta_\bk^\gamma [ \chi_\epsilon(\bk,\lambda) (a(\bk)-\lambda)^{-1}  ] = 0\]
    \textmd{which holds as $\chi_\epsilon(\cdot,\lambda)(a(\cdot)+\lambda)^{-1} \in \mathscr S(\Z^n,L(E))$ and all terms appearing in the discrete derivatives cancel.}

     Now note that for \textmd{$\bk\in\Z^n$}, $\eta\in\T^n$ and $\theta\in (0,1)$ the elementary inequality
    \[ |e^{i\textmd{\bk}\cdot\eta} - 1| \le 2 \textmd{|\bk|^\theta } |\eta|^\theta\]
    holds. From this and Lemma~\ref{4.5} \textmd{c) } we obtain
    \begin{align*}
    \big\| (e^{-i\eta}-1)^\gamma  & K_\epsilon(\eta,\lambda) \big\|_{L(E)} \textmd{\le C |\eta|^{\theta} \sum_{\bk\in\Z^n} |\bk|^\theta \langle\bk,\lambda\rangle^{-2m} \langle\bk\rangle^{m-n}}\\
    & = C  |\eta|^{\textmd{\theta}} \mu^\theta
     \sum_{\ell\in\mu^{-1}\Z^n} |\ell|^\theta \textmd{(1+\mu^2+|\mu\ell|^2)^{-m} } \langle\mu\ell\rangle^{m-n}.
     \end{align*}
     Choosing $\theta:= \theta_i$ for $|\gamma|=n+i,\, i\in\{0,1\}$, as in Lemma~\ref{4.4}, the last sum can be estimated by $C \mu^{-m-i}$. By Remark~\ref{4.6}, we see that
    \begin{align*}
     |\eta|^{n+i} \|K_\epsilon(\eta,\lambda)\|_{L(E)} &\le C\sum_{|\gamma|=n+i} \| (e^{-i\eta}-1)^\gamma K_\epsilon(\eta,\lambda)\|_{L(E)}\\
      & \le C |\eta|^{\theta_i} \mu^{-m-i+\theta_i}.
      \end{align*}
    Summation over $i\textmd{\in}\{0,1\}$ yields \eqref{4-11}. Now we integrate over $\eta\in\T^n$ and obtain
    \begin{align*}
      |\lambda|\,\|K_\epsilon(\cdot,\lambda)\|_{L(E)} & \le C \int_{\textmd{[-\pi,\pi]^n}} \frac{ \mu^{\theta_0} |\eta|^{\theta_0} + \mu^{\theta_1} |\eta|^{\theta_1}}{ |\eta|^n (1+\mu|\eta|)}\, d\eta \\
      & \le  C \int_{\R^n} \frac{|\xi|^{\theta_0} + |\xi|^{\theta_1}}{|\xi|^n(1+|\xi|)}\, d\xi <\infty.
    \end{align*}
    This shows \eqref{4-12}.

    We have seen above that \textmd{for fixed $(\eta,\lambda)$ we have}
    \[ (e^{-i\eta}-1)^\gamma K_\epsilon(\eta,\lambda) = \sum_{\bk\in\Z^n} \textmd{x_{\bk,\epsilon}(\eta,\lambda)}\]
    with \textmd{$x_{\bk,\epsilon}(\eta,\lambda) := (e^{i\bk\cdot\eta} -1) \Delta_{\bk}^\gamma [ \chi_\epsilon (\bk,\lambda)(a(\bk)+\lambda)^{-1}]$. } For
    \[ y_\bk(\eta,\lambda) :=  \textmd{|\eta|^{\theta}  |\bk|^\theta \langle\bk,\lambda\rangle^{-2m} \langle\bk\rangle^{m-n}}\quad (\bk\in\Z^n)\]
    we have shown \textmd{the uniform estimate (with respect to $\epsilon$)}
    \[ \|x_{\textmd{\bk,\epsilon}}(\eta,\lambda)\|_{L(E)}\le C y_{\bk}(\eta,\lambda).\]
\textmd{     As $\sum_{\bk\in\Z^n} y_\bk(\eta,\lambda) <\infty$, the sequence $(y_{\bk}(\eta,\lambda))_{\bk\in\Z^n}$ serves as a summable dominating sequence. Because of }  $\chi_\epsilon(\bk,\lambda)\to 1\;(\epsilon\searrow 0)$ for every fixed $\bk$ and $\lambda$, we get by dominated convergence the existence of
    \[ \tilde K(\eta,\lambda) := \lim_{\epsilon\searrow 0} (e^{-i\eta}-1)^\gamma K_\epsilon(\eta,\lambda)\]
    for every $\eta$ and $\lambda$. Setting $K(\eta,\lambda):= [(e^{-i\eta}-1)^\gamma]^{-1}\tilde K(\eta,\lambda)$ if $[\cdots]\not=0$ and $K(\eta,\lambda):= 0 $ else, we see that $K_\epsilon(\eta,\lambda)\to K(\eta,\lambda)$ pointwise almost everywhere.

    As the right-hand side of \eqref{4-11} is a dominating and integrable function which is independent of $\epsilon$, we see that $\|K(\cdot,\lambda)\|_{L(E)} = \lim_{\epsilon\to 0} \|K_\epsilon(\cdot,\lambda)\|_{L(E)}$ by dominated convergence again, which yields \eqref{4-12} for $K$ instead of $K_\epsilon$.
\end{proof}

\textmd{As in Theorem~\ref{4.2}, for $a\in\mathcal A$ we define }  $b_\lambda(\bk):= (a(\bk)+\lambda)^{-1}\;(\bk\in\Z^n)$.

\begin{theorem}
  \label{4.8}
  Let $k\in\N_0$ and $p\in[1,\infty]$. Then there exists a constant $M>0$ such that
  \begin{equation}
    \label{4-13}
    \|\textmd{\op[b_\lambda]}\|_{W_p^k(\T^n,E)} \le \frac{M}{|\lambda|}\,\|u\|_{W_p^k(\T^n,E)}
  \end{equation}
  for $u\in C^\infty(\T^n,E),\, \lambda\in\Sigma_{\pi/2,R}$ and $a\in\mathcal A$.   Therefore, in the case $p\in[1,\infty)$ we obtain $\textmd{\op[b_\lambda]}\in L(W_p^k(\T^n,E))$ with $\|\textmd{\op[b_\lambda]}\|_{L(W_p^k(\T^n,E))}\le \frac M{|\lambda|}$.
\end{theorem}

\begin{proof}
  Let $\alpha\in\N_0^n$ with $|\alpha|\le k$, $\lambda\in\Sigma_{\pi/2,R}$, $u\in C^\infty(\T^n,E)$, and $x\in\T^n$. Then (see Remark~\ref{3.4})
  \[ \partial^\alpha[b_\lambda(D)u](x) = \ossumint e^{i\bk\cdot\eta} (a(\bk)+\lambda)^{-1} (\partial^\alpha u)(x-\eta) d\eta.\]
  The arguments given in Remark~\ref{3.4} a) show that we can write
  \[ \partial^\alpha\big(\textmd{\op[b_\lambda]} u\big)(x) = \lim_{\epsilon\searrow 0} \textmd{\sumint} e^{i\bk\cdot\eta} \chi_\epsilon(\bk,\textmd{\lambda}) (a(\bk)+\lambda)^{-1} (\partial^\alpha u)(x-\eta) d\eta\]
  where $\chi_\epsilon$ is defined as in Lemma~\ref{4.5} b). This yields
  \begin{align}
    \partial^\alpha\big(\textmd{\op[b_\lambda]} u\big)(x) & = \lim_{\epsilon\searrow 0} \int_{\T^n} \Big( \sum_{\bk\in\Z^n} e^{i\bk\cdot\eta} \chi_\epsilon(\bk,\textmd{\lambda}) (a(\bk)+\lambda)^{-1} \Big) (\partial^\alpha u)(x-\eta)d\eta\nonumber\\
    & = \lim_{\epsilon\searrow 0} \int_{\T^n} K_\epsilon(\eta,\lambda)(\partial^\alpha u)(x-\eta)d\eta\label{4-14}
  \end{align}
  with $K_\epsilon$ from Lemma~\ref{4.7}. From \eqref{4-14}, Lemma~\ref{4.7} and dominated convergence, we get
  \[    \partial^\alpha [b_\lambda(D) u](x)  = \int_{\T^n} K(\eta,\lambda)(\partial^\alpha u)(x-\eta)d\eta = \big( K(\cdot,\lambda)\ast (\partial^\alpha u)\big)_{\T^n} (x).\]
  Since $\partial^\alpha u\in L^p(\T^n,E)$, we have $(K(\cdot,\lambda)\ast \partial^\alpha u)_{\T^n}\in L^p(\T^n,E)$
  and
  \begin{align*}
   \|\partial^\alpha [b_\lambda(D) u]\|_{L^p(\T^n,E)} & \le \|K(\cdot,\lambda)\|_{L^1(\T^n,L(E))} \|\partial^\alpha u\|_{L^p(\T^n,E)}\\
   &  \le \frac M{|\lambda|} \, \|u\|_{W_p^k(\T^n,E)}
   \end{align*}
  due to Lemma~\ref{3.10} b) and Lemma~\ref{4.7}. This proves \eqref{4-13}. If $p<\infty$, then $C^\infty(\T^n,E)$ is dense in $W_p^k(\T^n,E)$ which gives the estimate on $\textmd{\op[b_\lambda]}\in L(W_p^k(\T^n,E))$ and the estimate on its norm.
\end{proof}

\textmd{Now we are able to show that parabolic symbols lead to operators in the Sobolev spaces $W_p^k(\T^n,E)$ which are generators of analytic semigroups. For  $k\in\N_0$ and $p\in[1,\infty]$, we define the $W_p^k(\T^n,E)$-realization $A_{a,k}$  of the symbol $a\in \mathcal A$ as  the unbounded operator given by
\begin{align*}
  D(A_{a,k}) & := \{ u\in W_p^k(\T^n,E): \op[a]\in W_p^k(\T^n,E)\},\\
  A_{a,k} u  &:= \op[a] u\quad (u\in D(A_{a,k})).
\end{align*}
Note that due to the embedding $W_p^k(\T^n,E)\subset B_{p,\infty}^k(\T^n,E)$, the operator $\op[a]$ is well-defined on $W_p^k(\T^n,E)$.
}

\begin{remark}
  \label{4.9}
  For $\lambda\in \Sigma_{\pi/2,R}$ and $a\in\mathcal A$, the operator $A_{a,k}+\lambda$ is invertible and
  \begin{equation}
    \label{4-15}
    (A_{a,k}+\lambda)^{-1} =\textmd{ \op[b_\lambda]}\big|_{W_p^k(\T^n,E)}.
  \end{equation}
  In fact, $\textmd{\op[a]}+\lambda\colon B_{p,\infty}^{k+m}(\T^n,E)\to B_{p,\infty}^k(\T^n,E)$ is bijective by Theorem~\ref{4.2} which gives injectivity of $A_{a,k}+\lambda$.

  If $v\in W_p^k(\T^n,E)\subset B_{p,\infty}^k(\T^n,E)$, there exists a unique $u\in B_{p,\infty}^{k+m}(\T^n,E)$ such that $v=(\textmd{\op[a]}+\lambda) u$. As $B_{p,\infty}^{k+m}(\T^n,E)\hookrightarrow W_p^k(\T^n,E)$, we get
  $Au = v-\lambda u\in W_p^k(\T^n,E)$ and therefore $u\in D(A_{a,k})$. So $A_{a,k}+\lambda$ is surjective, too. Now \eqref{4-15} follows from Theorem~\ref{4.2} and the fact that $A_{a,k}+\lambda$ is a restriction of $\textmd{\op[a]}+\lambda\colon B_{p,\infty}^{k+m}(\T^n,E)\to B_{p,\infty}^k(\T^n,E)$.
\end{remark}

Now we are able to prove the main result of the present paper.

\textmd{\begin{theorem}
  \label{4.10}
  Let $a\in S^{m,\rho}$ be a parabolic symbol in the sense of Definition~\ref{4.1}.
  Let $k\in\N_0$ and $p\in [1,\infty]$, and let $A_{a,k}$ be the Sobolev space realization of the symbol $a$. Then there exist constants $M>0$ and $R>0$ such that $\Sigma_{\pi/2,R}\subset \rho(-A_{a,k})$ and
  \[ \| (A_{a,k}+\lambda)^{-1}\|_{L(W_p^k(\T^n,E))} \le \frac M{|\lambda|}\quad (\lambda\in \Sigma_{\pi/2,R}).\]
  In particular, $-A_{a,k}$ generates an analytic semigroup on $W_p^k(\T^n,E)$. If $p<\infty$, then the semigroup is strongly continuous.
\end{theorem}
}
\begin{proof}
\textmd{\textbf{(i)} First assume $p\in [1,\infty)$. } By Theorem~\ref{4.8} and Remark~\ref{4.9}, we obtain
  \[ \| (\lambda+A_{a,k})^{-1}\|_{L(W_p^k(\T^n,E))} \le \frac M {|\lambda|}\quad (\lambda\in \Sigma_{\pi/2,R}).\]
  Therefore, $-A_{k,p}$ generates a holomorphic semigroup on $W_p^k(\T^n,E)$, see \cite{lunardi95}, Proposition~2.1.11 and Proposition~2.1.1. As $C^\infty(\T^n,E)\subset D(A_{k,p})\subset W_p^k(\T^n,E)$ and $C^\infty(\T^n,E)$ is dense in $W_p^k(\T^n,E)$ for $p\in [1,\infty)$, the semigroup is even strongly continuous.

\textbf{(ii)} Now we consider the case $p=\infty$.
  The bijectivity of $A_{a,k}+\lambda\colon D(A_{a,k})\to W_\infty^k(\T^n,E)$ was already shown in Remark~\ref{4.9}. We choose $r_1,r_2\in\R$ with $r_1<\textmd{r_2}<k<r_1+m$ (e.g., $r_1 = k -\frac m2$) and set $r:= r_2-r_1$.

  Let $\psi\in\mathscr D(\R^n)$ with $0\le \psi\le 1$ and $\psi(0)=1$, and let $C_\psi:= \|\FR^{-1}\psi\|_{L^1(\R^n)}$. For $\epsilon>0$, we define $\psi_\epsilon:= \psi(\epsilon\,\cdot\,)$ and $\varphi_\epsilon(\bk):= \psi_\epsilon(\bk)\id_E\;(\bk\in\Z^n)$. Then it is easily seen that $\varphi_\epsilon\in S^{r,\rho}$, $\|1-\varphi_\epsilon\|_{S^{r,\rho}} \to 0\;(\epsilon\searrow 0)$ and
  \[ \|\FR^{-1} \varphi_\epsilon\|_{L^1(\R^n,L(E))} = \| \FR^{-1} \psi_\epsilon\|_{L^1(\R^n)} = \|\FR^{-1}\psi\|_{L^1(\R^n)} = C_\psi\quad (\epsilon\in (0,1)).\]
  For $u\in W_p^k(\T^n,E)$, we set $u_\epsilon:= \op[\varphi_\epsilon]u\;(\epsilon\in (0,1))$. Then $u_\epsilon\in C^\infty(\T^n,E)$, and Theorem~\ref{3.15} and the embedding $ W_\infty^k(\T^n,E)\subset B_{\infty,1}^{r_1}(\T^n,E)$ (see Remark~\ref{3.8}) yield
  \begin{align*}
   \|u-u_\epsilon\|_{B_{\infty,1}^{r_1}(\T^n,E)}
  & = \| \op[1-\varphi_\epsilon]u\|_{B_{\infty,1}^{r_1}(\T^n,E)}\\
    & \le C_1 \| 1-\varphi_\epsilon\|_{S^{r,\rho}} \|u\|_{B_{\infty,1}^{r_2}(\T^n,E)} \\
  &   \le C_2 \| 1-\varphi_\epsilon\|_{S^{r,\rho}} \|u\|_{W_\infty^k(\T^n,E)}\to 0
  \quad (\epsilon\searrow 0).
\end{align*}
With Remark~\ref{4.9}, Theorem~\ref{3.16}, and the embedding $ B_{\infty,1}^{r_1+m}(\T^n,E)\subset W_\infty^k(\T^n,E)$ we obtain
\begin{align*}
  \|(\lambda+A_{a,k})^{-1} (u-u_\epsilon)\|_{W_\infty^k(\T^n,E)} & = \|b_\lambda(D)(u-u_\epsilon)\|_{W_\infty^k(\T^n,E)} \\
  & \le C_3 \|b_\lambda(D) (u-u_\epsilon)\|_{B_{\infty,1}^{r_1+m}(\T^n,E)} \\
  & \le C_4 \|u-u_\epsilon\|_{B_{\infty,1}^{r_1}(\T^n,E)} \to 0\quad (\epsilon\searrow0).
\end{align*}
Thus, $(\lambda+A_{a,k})^{-1} u_\epsilon \to (\lambda+A_{a,k})u\;(\epsilon\searrow0)$ in $W_\infty^k(\T^n,E)$. Since
\begin{align*}
  \|u_\epsilon\|_{W_\infty^k(\T^n,E)} & = \|\op[\phi_\epsilon]u\|_{W_\infty^k(\T^n,E)} = (2\pi)^{-n/2} \big\|(\FR^{-1} \phi_\epsilon) \ast u\big\|_{W_\infty^k(\T^n,E)} \\
  & = (2\pi)^{-n/2} \max_{|\alpha|\le k} \big\| (\FR^{-1}\varphi_\epsilon)\ast \partial^\alpha u\big\|_{L^\infty(\T^n,E)} \\
  & \textmd{\le} (2\pi)^{-n/2} C_\psi \|u\|_{W_p^k(\T^n,E)}
\end{align*}
due to Lemma~\ref{3.10} a), we get from Theorem~\ref{4.8}
\[ (\lambda+A_{a,k})^{-1} u_\epsilon\|_{W_\infty^k(\T^n,E)}\le \frac M{|\lambda|}\, \|u_\epsilon\|_{W_\infty^k(\T^n,E)} \le \frac C{|\lambda|}\, \|u\|_{W_\infty^k(\T^n,E)}\,.\]
Taking $\epsilon\searrow 0$ on the left-hand side, we obtain the statement of the theorem.
\end{proof}

The above results on the generation of semigroups in $W_p^k(\T^n,E)$ allow us to solve non-autonomous Cauchy problems, based on the abstract results in \cite{amann95}, Chapter~IV. For this, \textmd{let $T>0$ } and assume $\mathcal A = \{a(t,\cdot): t\in [0,T]\}\subset S^{m,\rho}$ to be a family of operator-valued symbols on the torus. For $p\in [1,\infty)$ and $k\in\N_0$, we denote by $A_{a,k}(t)$ the $W_p^k(\T^n,E)$-realization of $a(t,\cdot)$. We study the toroidal Cauchy problem
\begin{equation}
\addtolength{\arraycolsep}{-0.3em}
  \label{4-16}
  \left\{
  \begin{array}{rcll}
    \partial_t u(t) + A_{a,k}(t) u(t) & = &f(t)  \quad(t\in (0,T]),\\
    u(0) & = &\textmd{u_0}.&
  \end{array}\right.
\end{equation}
\addtolength{\arraycolsep}{0.3em}
A function $u\in C^1((0,T], W_p^k(\T^n,E))\cap C([0,T], W_p^k(\T^n,E))$ is called a classical solution of \eqref{4-16} if $u(t)\in D(A_{a,k}(t))$ for all $t\in (0,T]$, $u(0)=u_0$, and if \eqref{4-16} holds for all $t\in (0,T]$.

\begin{theorem}
  \label{4.12}
  Let $k\in \N_0$, $p\in [1,\infty)$, $\alpha,\sigma\in (0,1)$, and assume
  \[ \big[ t\mapsto a(t,\cdot)\big]\in C^\alpha([0,T], S^{m,\rho}).\]
  Moreover, assume that there exist $\omega\ge 0$ and $\kappa>0$ such that $\mathcal A = \{ a(t,\cdot): t\in [0,T]\}\subset P_{\omega,\kappa} S^{m,\rho}$. Then for every $u_0\in W_p^k(\T^n,E)$ and every $f\in C^\sigma([0,T], W_p^k(\T^n,E))$ the Cauchy problem \eqref{4-16} has a unique classical solution.
\end{theorem}

\begin{proof}
  Using Theorem~\ref{4.8}, this follows from the abstract result on Cauchy problems, Theorem~2.5.1 of Chapter IV in \cite{amann95} analogously to the proof of Theorem~\ref{4.3} in \cite{barraza-denk-monzon14}.
\end{proof}

\bibliographystyle{abbrv}
\def\cprime{$'$}

\end{document}